\documentclass[12pt]{article}

\usepackage{amsthm,amsmath,amsfonts}
\usepackage{latexsym}
\usepackage[T1]{fontenc}
\usepackage{a4wide}
\usepackage{times}
\usepackage{array}
\usepackage{hhline}
\usepackage{graphicx}
\numberwithin{equation}{section}

\newtheorem{tw}{Theorem}
\newtheorem{lem}{Lemma}

\begin{document}

\title{Strong solutions to the Navier--Stokes--Fourier system with slip--inflow boundary conditions
\footnotetext{
\hfill\break}
\footnotetext{\textbf{Keywords.}  Steady Navier--Stokes--Fourier system; inflow boundary conditions; strong solution; small data; \hfill\break}}
\author{{\textsc{Tomasz Piasecki$^1$}} and \textsc{Milan Pokorn\'y$^2$}
}
\maketitle

\begin{center}
{\small {
{1.  Institute of Applied Mathematics and Mechanics}

{University of Warsaw,
ul. Banacha 2, 02-097 Warszawa, Poland}

{E-mail: {\tt tpiasecki@mimuw.edu.pl}}

{2. Mathematical Institute of Charles University}

{Faculty of Mathematics and Physics, Charles University in Prague}

{Sokolovsk\'a 83, 186 75 Praha 8, Czech Republic }

{E-mail: {\tt pokorny@karlin.mff.cuni.cz}}
}

}

\end{center}

\begin{abstract}
We consider a system of partial differential equations describing the steady flow of a compressible heat conducting Newtonian fluid in a three-dimensional channel with inflow and outflow part. We show the existence of a strong solution provided the data are close to a constant, but nontrivial flow with sufficiently large dissipation in the energy equation. 

\end{abstract}

\section{Introduction}
We investigate the stationary flow of a heat conducting compressible fluid in a cylindrical domain. The fluid is assumed to be Newtonian.
Then the flow is described by the stationary Navier-Stokes-Fourier (NSF) system:    
\begin{equation}  \label{main_system}
\begin{array}{lcr}
\rho v \cdot \nabla v - {\rm div}\, {\bf S}(\nabla v) + \nabla \pi(\rho,\theta) = \rho  f & \mbox{in} & \Omega,\\
{\rm div}\;(\rho v)=0 & \mbox{in} & \Omega,\\
{\rm div}\; (\rho E v) = \rho f \cdot v - {\rm div}\,(\pi v) + {\rm div}\, ({\bf S}(v)v) - {\rm div}\,q & \mbox{in} &\Omega,\\
({\bf S}(\nabla v) n)\cdot \tau_k + \alpha v\cdot \tau_k=b_k, \quad k=1,2 & \mbox{on} &\Gamma,\\
n \cdot v=d & \mbox{on} & \Gamma,\\
\rho=\rho_{in} & \mbox{on} & \Gamma_{in},\\
-q \cdot n + L(\theta-T) = g & \mbox{on} & \Gamma,
\end{array}
\end{equation}
where $\Omega = \Omega_0 \times [0,l]$ with $\Omega_0 \subset \mathbb{R}^2$ smooth,
$v$ is the velocity field of the fluid, $\rho$ is the density, $\theta$ is the absolute temperature. We assume that the  
pressure $\pi(\rho,\theta)$ is a twice continuously differentiable function on $\mathbb{R}^+ \times \mathbb{R}^+$ such that
\begin{equation} \label{1.2a}
\begin{array}{rcl}
\pi(1,T_0) &=& p_0 >0, \\
\partial_\rho \pi (1,T_0) &=&  p_1 >0, \\
\partial_\theta \pi (1,T_0) &=&  p_2 >0.  
\end{array}
\end{equation}
We consider the Newtonian compressible fluid, i.e. the stress tensor has the form
\begin{equation} \label{1.2b}
{\bf S}(\nabla v) = \mu \Big(\nabla v + \nabla^T v - \frac 23 {\rm div}\, v {\bf I}\Big) + \lambda {\rm div}\, v {\bf I},
\end{equation}
where the constant viscosities  $\mu$ and $\lambda$ fulfill $\mu >0$ and $\lambda \geq 0$.

Further, the friction coefficient $\alpha >0$.  Finally, $q=-\kappa \nabla \theta$ is the heat flux. Note that we could treat the situation $\kappa = \kappa (\rho, \theta)$, $\mu = \mu(\rho, \theta)$ and $\lambda = \lambda(\rho,\theta)$ with suitable assumptions on these functions. It would only lead to further complications, therefore we take them rather constant not to hide the main ideas by too many technicalities.

%
%

The boundary of the domain is divided in a natural way into the inflow part $\Gamma_{in}$,
the outflow part $\Gamma_{out}$ and the impermeable wall $\Gamma_0$. More precisely, we define
\begin{displaymath}
\begin{array}{c}
\Gamma_{in} = \{ x \in \Gamma: d < 0 \},\\
\Gamma_{out} = \{ x \in \Gamma: d > 0 \},\\ 
\Gamma_0 = \{ x \in \Gamma: d = 0 \}.
\end{array}
\end{displaymath}
We consider a channel-like flow, i.e. we further assume that $\Gamma_{in} = \Omega_0\times \{0\}$, $\Gamma_{out} = \Omega_0 \times \{l\}$ and $\Gamma_0 = \partial \Omega_0 \times [0,l]$. 
We set $L=0$ on $\Gamma_{in} \cup \Gamma_{out}$ and $L$ to be a given positive constant on the wall $\Gamma_0$.
We set $g=0$ on $\Gamma_0$ but we admit $g \neq 0$ on $\Gamma_{in} \cup \Gamma_{out}$ (see fig.\ref{rys1}).
\begin{figure}[htb]
\begin{center}
\includegraphics[width = 0.5\textwidth]{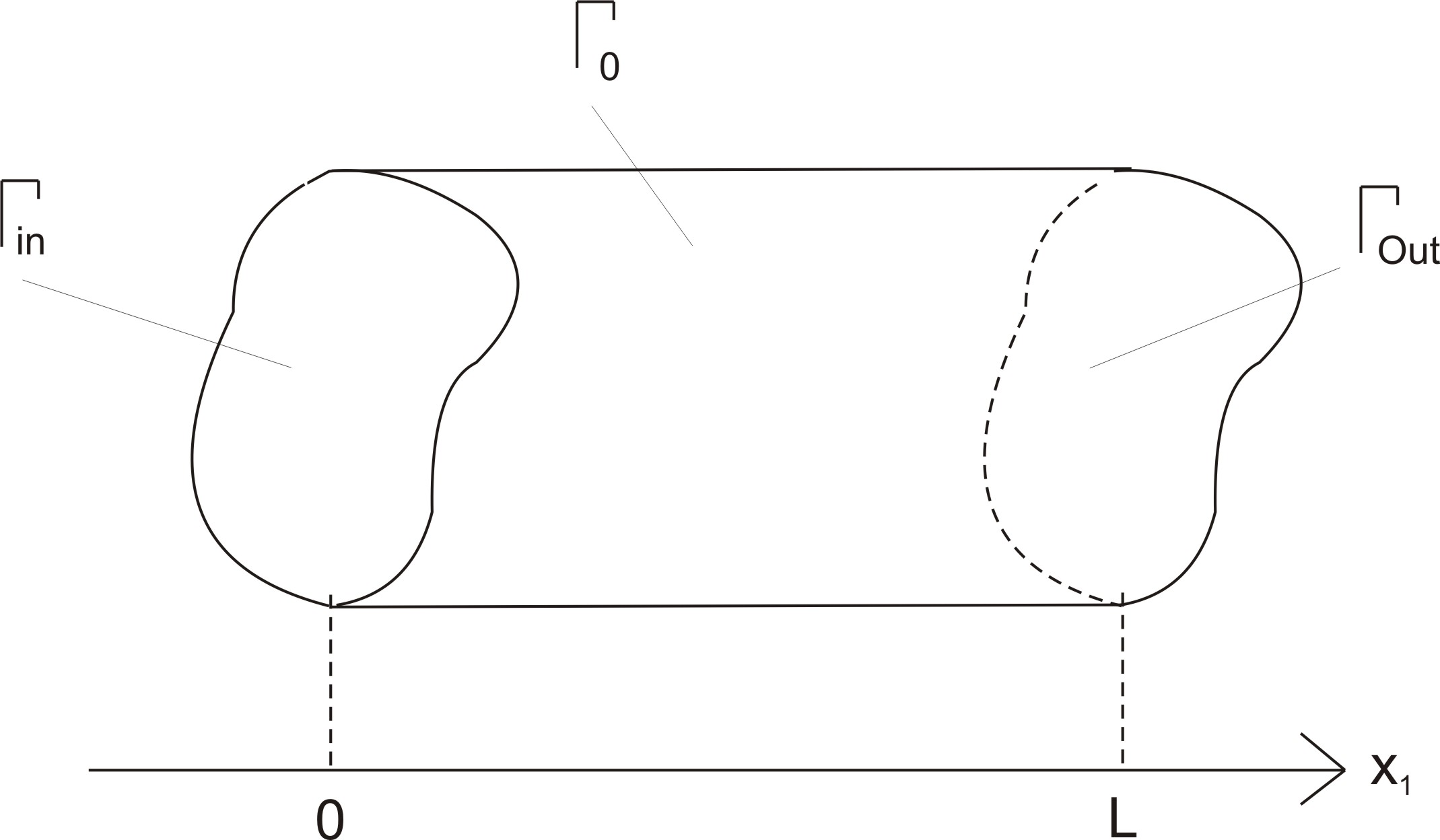}
\caption{The domain}      \label{rys1}
\end{center}
\end{figure}
Under such choice, on the inflow and outflow part the boundary condition for the temperature
reduces to the inhomogeneous Neumann condition. 
From the point of view of modeling these parts of the boundary are artificial and we can measure the
parameters of the flow, what can be reflected in different values of $g$.
The condition on $\Gamma_0$ means that 
the flux of the temperature through the wall is proportional to the difference 
between the outside temperature $T$ and the temperature of the fluid.  
The level of thermal insulation of the wall is given by $L$.
Existence and uniqueness of the special solution $\bar \theta$
under our choice of boundary conditions is discussed in the next chapter.

%
%
The total energy is given by 
$E = \frac{1}{2}|u|^2+e(\rho,\theta)$, where the internal energy $e$ is given by
\begin{equation} \label{def_e}
e(\rho,\theta) = e_\pi(\rho,\theta) + c_v \theta 
\end{equation}
with
\begin{equation} \label{1.3a}
\partial_\rho e_\pi (\rho,\theta) = \frac{1}{\rho^2}\big(\pi(\rho,\theta) - \theta \partial_\theta\pi(\rho,\theta)\big),
\end{equation}
i.e. the internal energy fulfills the Maxwell relation. Without loss of generality we set in \eqref{def_e} $c_v = 1$. Equivalently we can consider system (\ref{main_system}) with total energy balance 
(\ref{main_system})$_3$ replaced with the internal energy balance
\begin{equation} \label{int_en_bal}
{\rm div} (\rho e v) - \kappa \Delta \theta = {\bf S}(\nabla v) : \nabla v - \pi {\rm div} v. 
\end{equation} 
We also set
\begin{equation} \label{1.6a}
\partial_\rho e_\pi (1,T_0) = e_1, \qquad \partial_\theta e_\pi(1,T_0) = e_2.
\end{equation}
%
%

{\bf Known results}

The existence theory for the steady compressible Navier--Stokes equations in the framework of strong solutions was intensively studied in eighties, see e.g. \cite{NoPa}, \cite{NoPi} and many other papers. In the beginning of this century, inspired by the results of P.L. Lions (see \cite{Li}),  the attention turned rather towards the weak solutions, i.e. towards the existence theory without any smallness assumption on the data. The best results in this direction so far can be found in \cite{JeNo} (slip boundary conditions) and \cite{JiZh} (space periodic situation). Note, however, that the theory is not able to treat the situation with nonzero inflow/outflow conditions.   

Strong solutions with inhomogeneous boundary data, which are in our scope of interest in this paper,
have been considered for the first time in \cite{VZ} in Hilbert spaces and later in $L_p$ framework
(\cite{Kw1}, \cite{Kw2}, \cite{PRS1}).

In \cite{TP1} a constant flow in the direction of the axis of the cylinder is investigated and existence
of a strong solutions in its vicinity is shown in the barotropic case.

Concerning the NSF system, the existence of weak solutions was studied much later, see \cite{MuPo1}, \cite{MuPo2}. The best result in this direction can be found in \cite{JeNoPo}. Note, however, that similarly to the barotropic case, all these results do not allow to treat the inflow/outflow problems.
 
The theory of strong solutions for the NSF system describing the thermal effects is much less developed,
to our knowledge the only result in $L_p$ framework is the paper by Beirao da Veiga \cite{BdV} 
where existence of strong solutions in a vicinity of a constant flow with zero velocity is shown.

{\bf Special solution.} 
Here we are interested in showing existence of strong solutions, 
i.e. we are looking for $(v,\rho,\theta) \in W^2_p(\Omega) \times W^1_p(\Omega) \times W^2_p(\Omega)$,
solving system (\ref{main_system}) so that the solution is close to a given special solution. Hence we arrive at a problem
with some smallness of the data. Even though smallness obviously facilitates the analysis, 
the results of this kind are
still missing in the theory of compressible flow, especially for heat-conducting fluids (see the overview above).
This is mostly due to high complexity of the system (\ref{main_system})$_{1-3}$, 
in particular due to its mixed, elliptic (or parabolic
in time dependent case) and hyperbolic character. 
Investigation of stationary solutions close to given laminar flows is of particular importance
if we want to investigate stability of special solutions, which is quite well developed for weak solutions
but so far almost not investigated in strong solutions framework (see \cite{NoS} for an overview of results). 

The simplest example of a special solution to perturb is a solution with zero velocity and constant density
and temperature. Such flow has been investigated in \cite{BdV} with homogeneous Dirichlet boundary conditions.
From the point of view of applications it is important to investigate flows with large velocity,
which leads in a natural way to inhomogeneous boundary data. 
In a cylindrical domain a natural example of such solution is a constant flow in the direction
of the axis of the cylinder:
\begin{equation} 
\bar v = [1,0,0], \quad \bar \rho = 1.
\end{equation}
In other words, we can state that such a flow 
is natural to be investigated as a special solution to the system (\ref{main_system})
in our domain $\Omega$
 
%
 
As for the temperature of the basic solution the simplest choice would be to
assume it constant.  
It seems however quite artificial with the flow $(\bar v, \bar \rho)$ defined above. 
It is more interesting from physical point of view to consider
a flow with nontrivial temperature function $\bar \theta$ which corresponds 
to the solution $(\bar v, \bar \rho)$  as a solution to corresponding energy balance equation.    
On the boundary we assume the temperature of the form $T = T_0+T_1$
with constant, possibly large $T_0$ and small variation $T_1$.
Substituting $(\bar v, \bar \rho)$ to the internal energy balance equation we see that
$\bar \theta$ should satisfy the system
\begin{equation} \label{eqn_theta}
\begin{array}{lr}
(1+ e_2) \partial_{x_1}\overline {\theta} - \kappa \Delta \overline{ \theta} = 0 & \mbox{ in } \Omega,\\
\kappa \frac{\partial \overline{ \theta}}{\partial n} + L(\overline{ \theta} - T_0 - T_1) = g & \mbox{ on } \Gamma. 
\end{array}
\end{equation} 
Now we write $\overline{ \theta} = \overline{\theta}_1 + \overline{ \theta}_2 = T_0 + (\overline{ \theta}_0 - T_0) + \overline{ \theta}_1$, 
where $\overline{ \theta}_0$ solves 
\begin{equation} \label{eqn_tilde_theta}
\begin{array}{lr}
- \kappa \Delta \overline{ \theta}_0 = c_g & \mbox{ in } \Omega, \\[2pt]
\kappa \frac{\partial \overline{ \theta}_0}{\partial n}|_{\Gamma} = g & \mbox{ on } \Gamma,\\[2pt]
\frac{1}{|\Omega|}\int_{\Omega} \overline{ \theta}_0 \,dx = T_0,
\end{array}
\end{equation}
and
\begin{equation} \label{def_cg}
c_g = \frac{1}{|\Omega|} \int_{\Gamma} g \,dS.
\end{equation}

\noindent
The existence of a solution $\overline{ \theta}_0 \in W^2_p$ to (\ref{eqn_tilde_theta}) is a classical elliptic
result, we only need to apply symmetry to deal with corner singularities. The details
are given in Lemma \ref{lem_Neumann}. In particular $\overline{ \theta}_0$ satisfies 
\begin{equation} \label{est_tilde_theta}
\|\nabla^2 \overline{\theta}_0\|_{L_p} + \|\nabla \overline{ \theta}_0\|_{L_p} + \|\overline {\theta}_0-T_0\|_{L_p} 
\leq C [ \|g\|_{W^{1-1/p}_p(\Gamma)} + \|c_g\|_{L_p(\Gamma)} ] \leq C \|g\|_{W^{1-1/p}_p(\Gamma)}.
\end{equation}
 
Now we see that $\overline{ \theta}_1$ satisfies
\begin{equation} \label{eqn_theta1}
\begin{array}{lr}
 (1+e_2)\partial_{x_1}\overline{ \theta}_1 - \kappa \Delta \overline{ \theta}_1 = - (1+ e_2)\partial_{x_1}\overline{ \theta}_0 - c_g & \mbox{ in } \Omega,\\
\kappa \frac{\partial \overline{ \theta}_1}{\partial n} + L(\overline{ \theta}_1 - (T_1+T_0-\overline{\theta}_0)) = 0 & \mbox{ on } \Gamma.
\end{array}
\end{equation} 
The existence and uniqueness of solution to the above system is discussed in Subsection \ref{sec_theta1}.


Our goal is to show the existence of a solution to (\ref{main_system}) 
close to $(\bar v, \bar \rho, \bar \theta)$. 
Hence it is convenient to introduce the following quantity to measure the distance
of the data from the special solution:
\begin{equation} \label{D0}
\begin{array}{rcl}
D_0 &=& \|f\|_{L_p(\Omega)} + \Sigma_{i=1}^2\|b_i - \alpha \tau_i^{(1)}\|_{W^{1-1/p}_p(\Gamma)} + \|d - n^{(1)}\|_{W^{2-1/p}_p(\Gamma)}\\
&+& \|\rho_{in}-1\|_{W^{1}_p(\Gamma_{in})} + \|g\|_{W^{1-1/p}_p(\Gamma)} + \|T_1\|_{W^{1-1/p}_p(\Gamma)}.
\end{array}
\end{equation} 
We are now in a position to formulate our main result.
\begin{tw} \label{main_thm}
Assume that $D_0$ defined in (\ref{D0}) is small enough, $\kappa$ is large enough, 
$L$ is large enough on $\Gamma_0$, $\alpha$ large enough on $\Gamma_{in}$
and $p>3$. Then there exists 
a solution $(v,\rho,\theta)$ to system (\ref{main_system}) such that 
\begin{equation} \label{est_main}
\|v - \bar v\|_{W^2_p} + \|\rho - \bar \rho\|_{W^1_p} + \|\theta-\bar \theta\|_{W^2_p} \leq E(D_0).
\end{equation}
This solution is unique in the class of solutions satisfying (\ref{est_main}).
\end{tw}

{\bf Notation.} We use standard notation for the Sobolev spaces, i.e. $W^1_p(\Omega)$. For the spaces defined on $\Omega$ we will skip the domain;
for example we write $L_2$ instead of $L_2(\Omega)$. A generic constant that is controlled (but not necessarily small) 
will be denoted by $C$, whereas $E$ will denote a small constant (i.e. which can be arbitrarily small for the data small
enough). We will also use the space $L_\infty((0,L),L_2(\Omega_0))$, for simplicity we denote it by $L_\infty(L_2)$. 

\section{Preliminaries}
\subsection{Auxiliary results}
In this section we collect some standard results which we use throughout the paper.
We start with the Sobolev imbedding theorem, let us state here the particular case which is used
throughout the paper (the general version with the proof can be found in \cite[Theorem 5.4]{Ad1})
\begin{lem}
Let $\Omega$ be defined as above and let $f \in W^1_p (\Omega)$ with $p>3$. Then $f \in L_\infty(\Omega)$
and there exist $C=C(p,\Omega)$ s.t.
$$
\|f\|_{L_\infty} \leq C \|f\|_{W^1_p}.
$$ 
\end{lem}
Next we recall the Poincar\'e inequality. Among its different versions we will need the following two.
\begin{lem}
Let $\Gamma_1$ and $\Gamma_2$ denote two nontrivial parts of the boundary 
$\Gamma = \partial \Omega$ and $u \in W^1_2(\Omega)$. Then
\begin{equation} \label{Poin1}
\|u\|_{L_2(\Gamma_1)}^2 \leq C_P \big[ \|u\|_{L_2(\Gamma_2)}^2 + \|\nabla u\|_{L_2(\Omega)}^2\big]  
\end{equation}
and
\begin{equation} \label{Poin2}
\|u\|_{L_2(\Omega)}^2 \leq C_P \big[ \|u\|_{L_2(\Gamma_1)}^2 + \|\nabla u\|_{L_2(\Omega)}^2\big],  
\end{equation}
where $C_P=C_P(\Omega)$.
\end{lem}

\begin{lem} (interpolation inequality): \label{lem_int} \\
Let $2 < p <\infty$. Then for all $\epsilon >0 \quad \exists C(\epsilon,p,\Omega)$ such that $\forall f \in W^1_p(\Omega)$:
\begin{equation}  \label{int1}
\|f\|_{L_p} \leq \epsilon \|\nabla f\|_{L_p} + C \, \|f\|_{L_2}.
\end{equation}
\end{lem}
\noindent
\begin{proof}
The interpolation inequality in the Lebesgue spaces 
and the imbedding $W^1_p \hookrightarrow L_{q}$ with $q=\infty$ for $p>3$, $q<\infty$, arbitrary for $p=3$ and $q= \frac{3p}{3-p}$  for $p<3$ yield 
\begin{displaymath}  
\|f\|_{L_p} \leq C(p) \, \|f\|_{L_q}^{\theta} \, \|f\|_{L_2}^{1-\theta} \leq
C(p) (\|f\|_{L_p} + \|\nabla f\|_{L_p})^{\theta} \, \|f\|_{L_2}^{1-\theta}.
\end{displaymath}
Now application of the Young inequality leads to (\ref{int1}).
\end{proof}

Another result we use is the following version of the Korn inequality 
\begin{lem} \label{lem_Korn}
Let $\Omega \in C^{0,1}$, $\mu >0$ and if $\alpha =0$ then $\Omega$ is not axially symmetric. Then 
\begin{equation} \label{Korn}
\int_{\Omega}  \mu \big|\nabla u+ \nabla^T u -\frac 23 {\rm div}\, u {\bf I}\big|^2 \, dx + \int_{\Gamma} (\alpha (u \cdot \tau)^2 + d^2)  \,dS \geq C \, \|u\|_{W^1_2}^2.
\end{equation}
\end{lem}
\begin{proof}
It is a straightforward generalization of \cite[Lemma 4]{JeNo}.
\end{proof}
We also apply the following well-known result in finite dimensional Hilbert spaces
(the proof can be found in \cite{Te}):
\begin{lem}  \label{lem_P}
Let $X$ be a finite dimensional Hilbert space and let $P$: $X \to X$ be a continuous operator
satisfying
\begin{equation}  \label{lem_P_1}
\exists M>0: \quad (P(\xi),\xi) > 0 \quad \textrm{for} \quad \|\xi\| = M.
\end{equation}
Then
$
\exists \xi^*: \quad \|\xi^*\| \leq M \quad \textrm{and} \quad P(\xi^*) = 0.
$
\end{lem}

\noindent
Our proofs will be based on the $L_p$ regularity of elliptic problems with Neumann
and slip boundary conditions. These are classical results, in case of cylindrical domain
we only have to deal with the singularities of the boundary, for the sake of completeness we give the proofs.
The first result concerns the Neumann problem. 
\begin{lem} \label{lem_Neumann}
Let $1<p<\infty$, $f \in L_p(\Omega)$, $d \in W^{1-1/p}_p(\Gamma)$ and 
$\int_{\Omega} f \,dx = \int_{\Gamma} d \,dS$. Let $C_0 \in {\mathbb R}$. 
Then there exists a solution $u \in W^2_p(\Omega)$ to the problem 
\begin{equation}
\begin{array}{lcr}
\Delta u = f & \mbox{ in } & \Omega, \\[2pt]
\frac{\partial u}{\partial n} = d & \mbox{ on } & \Gamma, \\[2pt]
\frac{1}{|\Omega|}\int_\Omega u \, dx & = C_0,
\end{array}
\end{equation}
and 
\begin{equation}
\|u-C_0\|_{W^2_p} \leq C \, \big[ \|f\|_{L_p} + \|d\|_{W^{1-1/p}_p(\Gamma)} \big].
\end{equation}
Moreover, this solution is unique in the given regularity class.
\end{lem}   
\begin{proof} In fact, the only difficulty we have to overcome are the
singularities of the boundary.
 If $d|_{\Gamma_{in} \cup \Gamma_{out}} = 0$, then we can simply extend the weak solution
using symmetric reflection that preserves the homogeneous Neumann condition. 
In the neighbourhood of $\Gamma_{in}$ it will be defined as
\begin{equation}
E^v_s(u)(x) = \left\{ \begin{array}{c}
u(x), \quad x_1 \geq 0, \\
u(\tilde x), \quad x_1 < 0,
\end{array} \right.
\end{equation}
where $x = (x_1,x_2,x_3)$ and $\tilde x= (-x_1,x_2,x_3)$.
        
For $E^v_s(u)$ we obtain a Neumann problem in a domain with smooth boundary that we solve using classical
elliptic theory.

If $d$ does not vanish on $\Gamma_{in} \cup \Gamma_{out}$, we reduce the problem to the
previous case. Let us consider $\Gamma_{in}$, 
on $\Gamma_{out}$ we proceed the same way. 
Introducing an extension $u_0 \in W^2_p(Q)$ such that 
$\frac{\partial u_0}{\partial n}|_{\Gamma_{in}} = d$ and $\|u_0\|_{W^2_p} \leq C \|d\|_{W^{1-1/p}_p(\Gamma)}$, we consider $\tilde u = u - u_0$
that satisfies
\begin{displaymath}
\begin{array}{c}
\Delta \tilde u = \tilde f, \\
\frac{\partial \tilde u}{\partial n}|_{\Gamma} = 0,
\end{array}
\end{displaymath}
where 
$\|\tilde f\|_{L_p} \leq C \, [ \|f\|_{L_p} + \|d\|_{W^{1-1/p}_p(\Gamma)}]$.
\end{proof}

The second result concerns the Lam\'e system with slip boundary conditions:
\begin{equation} \label{Lam\'e}
\begin{array}{lcr}
-\mu {\rm div}\, (\nabla u + \nabla^T u -{\rm div}\, u {\bf I}) - \lambda \nabla {\rm div}\, u = F & \mbox{in} & \Omega, \\
2\mu ({\bf D}( u) n)\cdot \tau_i + \alpha \ u \cdot \tau_i = B_i, \quad i=1,2
&\mbox{on} & \Gamma, \\
n\cdot  u = 0 & \mbox{on} & \Gamma.\\
\end{array}
\end{equation}
The following lemma gives existence of a solution and the maximal regularity elliptic estimate
that we apply in our proofs.
\begin{lem} \label{lem_Lame}
Under the assumptions of Theorem \ref{main_thm} 
 there exists a unique $u \in W^2_p(\Omega)$ solving 
(\ref{Lam\'e}) with
\begin{equation} \label{est_lame_3d}
\|u\|_{W^2_p} \leq C \big[\|F\|_{L_p} + \sum_{i=1}^2\|B_i\|_{W^{1-1/p}_p(\Gamma)}\big].
\end{equation}
\end{lem}
\begin{proof} 
Under the assumptions on $\mu$ and $\lambda$ (\ref{Lam\'e}) is elliptic so we easily
get a weak solution. 
The only problem we encounter showing regularity of the weak solution 
are the singularities of the boundary
on the junctions of $\Gamma_0$ with $\Gamma_{in}$ and $\Gamma_{out}$.
These singularities can again be dealt  using the symmetry arguments.
Hence we define the operator $E^v_{as}$ extending a vector field defined
for $\{x: x_1 \geq 0\}$ on the whole space as 
\begin{equation}
E^v_{as}(u)(x) = \left\{ \begin{array}{c}
u(x), \quad x_1 \geq 0, \\
\tilde u(\tilde x), \quad x_1 < 0,
\end{array} \right.
\end{equation}        
where $\tilde x$ is as above and 
$\tilde u(\tilde x) = [-u^1(x),u^2(x),u^3(x)]$.
Then we have 
\begin{equation}
\begin{array}{c}
-{\rm div}\, \big(\nabla E^v_{as}(v) + \nabla^T E^v_{as}(v) - \frac 23 {\rm div}\, v {\bf I}\big)  + \lambda \nabla {\rm div}\, E^v_{as}(v) \\
= -{\rm div}\, \big(\nabla v + \nabla^T v - \frac 23 {\rm div}\, v {\bf I}\big)  + \lambda \nabla {\rm div}\, v,
\end{array}
\end{equation}
and on the plane ${x_1=0}$ the extension $E^v_{as}$ preserves the slip boundary conditions, i.e.
$$
2\mu  ({\bf D}( E^v_{as}(v) )n) \cdot \tau_i +\alpha (\ E^v_{as}(v) \cdot \tau_i) = 2 \mu  ({\bf D}( v )n) \cdot \tau_i +\alpha (v \cdot \tau_i) \\
$$
and $n \cdot E^v_{as}(v) = n \cdot  v$. Now we can show higher regularity of $E^v_{as}(u)$, extension
of the weak solution to (\ref{Lam\'e}), and since $E^v_{as}$ preserves the boundary conditions
we conclude that $u$ is a strong solution to (\ref{Lam\'e}) satisfying the boundary conditions
on $\Gamma_{in}$. Application of analogous antisymmetric extension on $\Gamma_{out}$
completes the proof. 
\end{proof}

\subsection{Remarks on the special solution} \label{sec_theta1}

Our next aim is to construct solutions to \eqref{eqn_theta}. Recall that
$$
\overline{\theta} = T_0 + (\overline{\theta}_0-T_0) + \overline{\theta}_1,
$$
where $\overline{\theta}_0$ solves \eqref{eqn_tilde_theta}. Therefore by Lemma \ref{lem_Neumann} we have
\begin{equation} \label{2.12b}
\kappa \|\overline{\theta}_0 -T_0\|_{W^2_p} \leq C (c_g + \|g\|_{W^{1-1/p}_p(\Gamma)}) \leq C \|g\|_{W^{1-1/p}_{p}(\Gamma)},
\end{equation}
as well as by the Lax-Milgram theorem
\begin{equation} \label{2.12bb}
\kappa \|\overline{\theta}_0 -T_0\|_{W^1_2} \leq  C \|g\|_{L_{2}(\Gamma)}.
\end{equation}

Next for $\overline{\theta}_1$ we have problem (\ref{eqn_theta1}). Our aim is to show existence of a weak solution to this problem as well as its regularity. As the problem is linear, the existence  and uniqueness of a weak solution is direct consequence of corresponding a priori estimates.  First we show estimates in $W^1_2$, then in $W^2_p$. Since we will require $\kappa$ and $L$ large, we have to control how all constants depend on them. For the sake of simplicity we assume that $\kappa$ and $L$ are approximately of the same size.

Multiplying  (\ref{eqn_theta1}) by $\theta_1$ and integrating over $\Omega$ we get
$$
\begin{array}{c}
\kappa \int_{\Omega} |\nabla \overline{\theta}_1|^2 \,dx + \int_{\Gamma_0} L \overline{\theta}_1^2 \,dS + \frac{1}{2}(1+e_2)\int_{\Gamma_{out}} \overline{\theta}_1^2 \,dS \\
= \frac{1}{2}(1+e_2) \int_{\Gamma_{in}} \overline{\theta}_1^2 \,dS + \int_{\Gamma_0} L \overline{\theta}_1 (T_1+T_0-\overline{\theta}_0) \,dS
- \int_{\Omega} ((1+e_2)\partial_{x_1}\overline{ \theta}_0 + c_g) \overline{\theta}_1 \,dx. 
\end{array}
$$
All terms on the l.h.s. are nonnegative. Applying the Poincar\'e inequality to the first term on the r.h.s. and the Young inequality to the other two ones we get
\begin{displaymath}
\begin{array}{c}
\kappa \int_{\Omega} |\nabla \overline{\theta}_1|^2 \,dx + L\int_{\Gamma_0} \overline{\theta}_1^2 \,dS  \\
\leq C_P \Big(\int_{\Gamma_0} \overline{\theta}_1^2 \,dS + \int_{\Omega} |\nabla \overline{\theta}_1|^2 \,dx \Big) + \frac L2 \int_{\Gamma_0} \overline{\theta}_1^2 \,dS + \frac \kappa 2\int_{\Omega} |\nabla \overline{\theta}_1|^2 \,dx \\ 
+ C \Big(L \int_{\Gamma_0} |T_1+ T_0-\overline{\theta}_0|^2 \,dS + \int_{\Omega} |\nabla \overline{\theta}_1|^2 \,dx+ c_g^2\Big). 
\end{array}
\end{displaymath}
Therefore we get, for $\kappa$ and $L$ sufficiently large with respect to $C_p$,
$$
\frac \kappa 4 \int_{\Omega} |\nabla \overline{\theta}_1|^2 \,dx + \frac L4\int_{\Gamma_0} \overline{\theta}_1^2 \,dS \leq 
C \big(\|g\|_{L_2(\Gamma)}^2 + L \|\overline{\theta}_0-T_0\|_{W^1_2}^2 + L \|T_1\|_{L_2(\Gamma)}^2 + \|\overline{\theta}_0-T_0\|_{L^2(\Gamma)}^2\big). 
$$ 
Hence, using also \eqref{2.12bb}
\begin{equation} \label{2.12c}
\|\overline{\theta}_1\|_{W^1_2} \leq C \big(\|g\|_{L_2(\Gamma)} + \|T_1\|_{L_2(\Gamma)} \big),
\end{equation}
where $C$ is a fixed constant independent of $\kappa$ and $L$, provided $\kappa \sim L$.
Uniqueness is a direct consequence of \eqref{2.12c}; if the data are zero, then the solution is zero as well.

Finally, proceeding similarly as above we will deduce the existence of strong solutions together with the estimates of the corresponding norms. First we rewrite  (\ref{eqn_theta1}) to the form
\begin{equation} \label{2.21}
\begin{array}{lr}
\kappa \Delta \overline{\theta}_1 = (1+e_2)\partial_{x_1}\overline{\theta}_1 + (1+e_2)\partial_{x_1}\overline{\theta}_0 + c_g & \mbox{ in } \Omega, \\[2pt]
\kappa \frac{\partial \overline{\theta}_1}{\partial n} = -L(\overline{\theta}_1 + \overline{\theta}_0 - T_0 -T_1 ) & \mbox{ on } \Gamma,
\end{array}
\end{equation}
and using Lemma \ref{lem_Neumann}
$$
\begin{array}{c}
\kappa \|\overline{\theta}_1\|_{W^2_p} \leq C \big(\|\overline{\theta}_1\|_{W^1_p} + \|\overline{\theta}_0\|_{W^1_p} + \|g\|_{L_2(\Gamma)} \\
+ L \|\overline{\theta}_1\|_{W^{1-1/p}_p(\Gamma)} + L \|T_1\|_{W^{1-1/p}_p(\Gamma)} + L \|\overline{\theta}_0-T_0\|_{W^{1-1/p}_p(\Gamma)}\big).
\end{array}
$$
Now, applying the Poincar\'e inequality in the form
$$
\|\overline{\theta}_1\|_{W^1_p} \leq \varepsilon  \|\overline{\theta}_1\|_{W^2_p} + C(\varepsilon)\|\overline{\theta}_1\|_{W^1_2}
$$
and using \eqref{2.12c} we end up with

\begin{lem} \label{l7}
Let $T_1 \in W^{1-1/p}_p(\Gamma)$, $g \in W^{1-1/p}_p(\Gamma)$, $p>3$ and let $\kappa$, $L$ be sufficiently large. Then there exist a unique weak solution $\theta_1$ 
to (\ref{eqn_theta1}) such that
\begin{equation} \label{est_theta1_1}
\|\overline{\theta}_1\|_{W^1_2} \leq C \big[ \|T_1\|_{L_2(\Gamma)} + \|g\|_{L_2(\Gamma)}\big].
\end{equation}
Moreover, the solution is strong, i.e. $\overline{\theta}_1 \in W^2_p$ and
\begin{equation} \label{est_theta1_2}
\|\overline{\theta}_1\|_{W^1_2} \leq C \big[ \|T_1\|_{W^{1-1/p}_p(\Gamma)} + \|g\|_{W^{1-1/p}_p(\Gamma)}\big].
\end{equation}
Here, the constants $C$ are independent of $\kappa$ and $L$.
\end{lem}

\subsection{System for perturbations} 
As we are interested in small perturbations of $(\bar v, \bar \rho, \bar \theta)$,
it is natural to introduce the perturbations as unknown functions.
For technical reasons it is better to get rid of inhomogeneity on the boundary.
Hence we construct $u_0 \in W^2_p$ such that $u_0 \cdot n|_{\Gamma} = d$. More precisely, we set $u_0 = \nabla \phi$,
where $\phi$ solves 
$$
\begin{array}{lcr}
\Delta \phi = \int_\Gamma d\, dS  & \mbox{in} & \Omega,\\
\frac{\partial \phi}{\partial n} = d & \mbox{on} & \Gamma.
\end{array}
$$
Next we introduce the perturbations
\begin{equation} \label{pert}
u = v - \bar v - u_0, \quad \sigma = \rho - \bar \rho, \quad \eta = \theta - \overline{ \theta} = \theta - \overline{\theta}_0 - \overline{\theta}_1.
\end{equation}
Replacing the total energy balance (\ref{main_system})$_3$ with the internal energy balance (\ref{int_en_bal})
we can rewrite (\ref{main_system}) as
\begin{eqnarray} \label{system}
\begin{array}{lcr}
\partial_{x_1}u - {\rm div}\, {\bf S}(\nabla u) +
p_1 \nabla \sigma + p_2 \nabla \eta = F(u,\sigma,\eta) & \mbox{in} & \Omega,\\
{\rm div}\, u + \sigma_{x_1} + (u+u_0) \cdot \nabla \sigma = G(u,\sigma)
& \mbox{in}& \Omega,\\
(1+ e_2) \partial_{x_1}\eta - \kappa \Delta \eta + T_0 p_2  {\rm div} \,u = H(u,\sigma,\eta) & \mbox{in} & \Omega,\\
({\bf S}(\nabla u)n)\cdot \tau_i +\alpha  u \cdot \tau_i = B_i, \quad i=1,2
&\mbox{on} & \Gamma, \\
n\cdot  u = 0 & \mbox{on} & \Gamma,\\
\sigma=\sigma_{in} & \mbox{on} & \Gamma_{in},\\
\kappa \frac{\partial \eta}{\partial n} + L \, \eta =0 & \mbox{on} & \Gamma,
\end{array}
\end{eqnarray}
where
\begin{equation}  \label{FGH}
\begin{array}{l}
F(u,\sigma,\eta) = (1+\sigma) f - p_2\nabla (\overline{\theta}_0 + \overline{\theta}_1) - \big(\partial_\rho \pi (\sigma+1, \overline{\theta}_0+\overline{\theta}_1+\eta)-p_1 \big)\nabla \sigma \\
\big(\partial_\theta \pi (\sigma+1, \overline{\theta}_0+\overline{\theta}_1+\eta)-p_2 \big)\nabla (\eta+\overline{\theta}_0 + \overline{\theta}_1) + {\rm div}\, {\bf S}(\nabla u_0) -\partial_{x_1}u_0 \\
- (u+u_0) \cdot \nabla (u+u_0)
- \sigma (\bar v + u+u_0) \cdot \nabla (u+u_0), 
\\[5pt]
G(u,\sigma) = -(\sigma + 1) \, {\rm div} \, u_0 - \sigma \, {\rm div} \,u,
\\ [5pt]
H(u,\sigma,\eta) = {\bf S}(\nabla(u+u_0)) \cdot \nabla (u+u_0) - T_0 \partial_\rho \pi(1+\sigma, \eta + \overline{\theta}_0 + \overline{\theta}_1) {\rm div}\,u_0 \\
 -(1+\sigma) (u+u_0) \cdot \nabla (\overline{\theta}_0 + \overline{\theta}_1 + \eta) +T_0 {\rm div} \,u  \big(p_2- \partial_\theta \pi (1+\sigma, \overline{\theta}_0+ \overline{\theta_1} +\eta)\big)
\\
-(\eta + (\overline{\theta}_0 -T_0) + \overline{\theta}_1) \partial_\theta \pi(1+\sigma, \overline{\theta}_0+ \overline{\theta_1} +\eta) {\rm div}\, (u+u_0) \\
 + \partial_{x_1} (\overline{\theta}_0 + \overline{\theta}_1 + \eta) \big(e_2- \partial_\theta e (1+\sigma, \overline{\theta}_0+ \overline{\theta_1} +\eta)\big),
\\[5pt]
B_i = b_i - \alpha \tau_i^{(1)} - ({\bf S}(\nabla u_0) n) \cdot \tau_i -\alpha u_0 \cdot \tau_i.
 \end{array}
\end{equation}
The above expressions result directly from substitution of (\ref{pert}) to (\ref{main_system}).
It is only worth to notice that to derive the expression for $H$ we have applied the continuity equation which yields 
\begin{equation} \label{div_id}
{\rm div} (\rho e v) = \rho v \cdot \nabla e = 
\rho v \cdot \nabla \rho \, \partial_\rho e + \rho v \cdot \nabla \theta \, \partial_\theta e.
\end{equation}
Now in (\ref{main_system})$_3$
the first term on the r.h.s. of (\ref{div_id}) can be combined  with the $\pi {\rm div}\,v$, what finally gives the above form of $H(\cdot)$.

Let us state at this early stage the following estimates 
on the quantities $F,G,H$ defined in (\ref{FGH}):
\begin{equation} \label{est_FGH}
\begin{array}{ll}
\|F(u,\sigma,\eta)\|_{L_p}  &\leq
C \big[ (\|u\|_{W^2_p} + \|\sigma\|_{W^1_p} + \|\eta\|_{W^2_p})^2  + (\|u\|_{W^2_p} + \|\sigma\|_{W^1_p})^3 + D_0(1+D_0)\big], \\
\\ \smallskip
\|G(u,\sigma)\|_{W^1_p} & \leq C \big[(\|u\|_{W^2_p} + \|\sigma\|_{W^1_p})^2 + D_0(1+D_0)\big],\\ 
\\ \smallskip
\|H(u,\sigma,\eta)\|_{L_p} 
&\leq C \big[ (\|u\|_{W^2_p} + \|\sigma\|_{W^1_p} + \|\eta\|_{W^2_p})^2  \\
&+ (\|u\|_{W^2_p} + \|\sigma\|_{W^1_p}+ \|\eta\|_{W^2_p})^3 + D_0(1+D_0)\big], \\
\end{array}
\end{equation}
where $D_0$ is defined in (\ref{D0}). 
The bounds (\ref{est_FGH}) are straightforward,
it is enough to observe that $F,G$ and $H$ are composed of terms which are either quadratic or cubic
with respect to the functions $u,\sigma,\eta$ or linear with respect to them multiplied by something which can be bounded by $D_0$.
Note that in the estimate of terms contained differences of the derivatives of $\pi$ or $e$ we use the fact fact that these functions are twice continuously differentiable and hence we may use the Taylor expansion theorem.
 
We sketch briefly the idea of the proof of Theorem 1, and hence the structure of the remainder
of the paper. In Subsection 2.4 we explain how to deal with the linearization of the continuity equation, i.e. equation \eqref{system}$_{2}$.
To show existence of solutions we can apply a method of successive approximations and look for the solution
as a limit of a sequence 
\begin{eqnarray} \label{system_seq}
\begin{array}{lcr}
\partial_{x_1} u^{n+1} - {\bf S}(\nabla u^{n+1}) +
p_1 \nabla \sigma^{n+1} + p_2 \nabla \eta^{n+1} = F(u^n,\sigma^n,\eta^n) & \mbox{in} & \Omega,\\
{\rm div} \, u^{n+1} + \sigma^{n+1}_{x_1} + (u^n+u_0) \cdot \nabla \sigma = G(u^n,\sigma^n)
& \mbox{in}& \Omega,\\
(1+ e_2) \partial_{x_1}\eta^{n+1} - \kappa \Delta \eta^{n+1} + T_0 p_2 {\rm div} u^{n+1} = H(u^n,\sigma^n,\eta^n) & \mbox{in} & \Omega,\\
 ({\bf S}(\nabla u^{n+1})n) \cdot \tau_i +\alpha u^{n+1} \cdot \tau_i = B_i, \quad i=1,2
&\mbox{on} & \Gamma, \\
n\cdot  u^{n+1} = 0 & \mbox{on} & \Gamma,\\
\sigma^{n+1}=\sigma_{in} & \mbox{on} & \Gamma_{in},\\
\kappa \frac{\partial \eta^{n+1}}{\partial n} + L \, \eta^{n+1} =0 & \mbox{on} & \Gamma.
\end{array}
\end{eqnarray}
In Section 3 we deal with the linear system (\ref{system_lin}) showing the a priori estimates. 
The first step is the energy estimate which is derived in quite a standard
way. The energy estimate is then applied to show the maximal regularity $L^p$ estimates, 
hence estimates in the space where we look for the solution (Lemma \ref{lem_est_lin}). 
It can be considered as the main result of Section 3.
In the end of Section 3 we solve the linear system, hence we show that the sequence (\ref{system_seq})
is well defined.    
In Section 4 we deal with the sequence of approximations. Using the estimates for the linear system
combined with (\ref{FGH}) we show it is bounded in $W^2_p \times W^1_p \times W^2_p$ and has a contraction
property in a weaker space $W^1_2 \times L_{\infty}(L_2) \times W^1_2$. 
An analogous estimate gives uniqueness in the class of solutions satisfying (\ref{est_main}). 
In other words we apply a generalization of the Banach fixed point theorem
to show the convergence of the sequence of approximations to the solution of (\ref{main_system}).
    
\subsection{Steady transport equation}
In this section we show the well posedness of the operator $S:L_2 \to L_{\infty}(L_2)$ defined as
\begin{equation} \label{def_S}
w = S(h) \iff \left\{
\begin{array}{lcr}
\partial_{x_1}w + U \cdot \nabla w = h & \textrm{in} & {\cal D'}(\Omega), \\
w = w_{in} & \textrm{on} & \Gamma_{in},
\end{array} \right.
\end{equation}
where $w$ is a weak solution to \eqref{def_S} in the sense as introduced in \eqref{weak2}.
The solvability of the above problem, to which we refer as a steady transport equation,
will be needed to show the energy estimate for the density in Section 3 and to solve the continuity equation
in Section 4. The result is given in the following
\begin{lem} \label{lem_S}
Let $h \in L_2(\Omega)$ and $w_{in} \in L_2(\Gamma_{in})$.
Let $U \in W^2_p(\Omega)$ be such that $\|U\|_{W^2_p}$ is small enough and $U \cdot n|_{\Gamma_0} = 0$.
Then there exists a unique solution $w=S(h) \in L_\infty(L_2)$ and
\begin{equation} \label{est_S}
\|S(h)\|_{L_{\infty}(L_2)} \leq C \, \big[\|w_{in}\|_{L_2(\Gamma_{in})} + \|h\|_{L_2(\Omega)}\big].
\end{equation}  
\end{lem}
The proof follows \cite{TP1}. For the sake of completeness we recall here the most important steps,
referring to \cite{TP1} for the missing details. The idea is to get rid of the term $U \cdot \nabla \sigma$ introducing
a change of variables
$x = \psi(z)$ satisfying the identity
\begin{equation} \label{change_id}
\partial_{z_1} = \partial_{x_1} + U \cdot \nabla_x.
\end{equation}
We construct the mapping $\psi$ in the following
\begin{lem}  \label{lem_change}
Let $\|U\|_{W^2_p}$ be small enough. Then there exists a set ${\cal O} \subset \mathbb{R}^3$
and a diffeomorphism $x=\psi(z)$ defined on ${\cal O}$ such that $\Omega = \psi({\cal O})$ and (\ref{change_id}) holds.
Moreover, if $z_n \to z$ and $\psi(z_n) \to \Gamma_0$ then $n^1(z) = 0$, where $n$ is the outward normal
to ${\cal O}$.

\end{lem}
Before we sketch the proof let us make one remark.
The last condition states that the first component of the normal to $\psi^{-1}(\Gamma_0)$ vanishes,
but since $\psi$ is defined only on ${\cal O}$ we formulate this condition using the limits.
It means simply that the image ${\cal O} = \psi^{-1}(\Omega)$ is also a cylinder with a flat wall. 

\noindent
\emph{Sketch of the proof of Lemma \ref{lem_change}.} 
The identity (\ref{change_id}) means that $\psi$ must satisfy
\begin{equation} \label{psi_z1}
\frac{\partial \psi^{1}}{\partial z_1} = 1 + U^1 (\psi), \quad
\frac{\partial \psi^{2}}{\partial z_1} = U^2 (\psi), \quad
\frac{\partial \psi^{3}}{\partial z_1} = U^3 (\psi).
\end{equation}
A natural condition is that $\psi(\Gamma_{in}) = \Gamma_{in}$.
Thus we can search for $\psi(z_1,z_2,z_3) = \psi_{z_2,z_3}(z_1)$,
where for all $(z_2,z_3)$ such that $(z_2,z_3,0) \in \Gamma_{in}$
the function $\psi_{z_2,z_3}(\cdot)$ is a solution to a system of ODE:
\begin{equation}  \label{ode}
\left\{ \begin{array}{l}
\partial_s \psi_{z_2,z_3}^1 = 1 + U^1(\psi_{z_2,z_3}), \quad
\partial_s \psi_{z_2,z_3}^2 = U^2(\psi_{z_2,z_3}), \quad
\partial_s \psi_{z_2,z_3}^3 = U^3(\psi_{z_2,z_3}),\\
\psi_{z_2,z_3}(0) = (0,z_2,z_3).
\end{array} \right.
\end{equation}
Using the regularity and smallness of $U$ we show existence for (\ref{ode}), 
hence we have $\psi$ defined on some ${\cal O}$ such that $\Omega = \psi({\cal O})$. 
The smallness of $U$ implies also that 
$\psi(z)=z+\psi_{\epsilon}(z)$, where $\|\psi_{\epsilon}\|_{W^1_\infty}$ is small and we 
conclude that $\psi$ is a diffeomorphism.
Let us denote $\phi = \psi^{-1}$. 
Now it is natural to define the
subsets of $\partial {\cal O}$ as $\partial {\cal O} = {\cal O}_{in} \cup {\cal O}_{out} \cup {\cal O}_0$ 
where ${\cal O}_{in} = \Gamma_{in}$,
${\cal O}_{out} = \{z: \; z = {\rm lim} \, \phi(x_n), \; x_n \to \Gamma_{out}\}$ and
${\cal O}_{0} = \{z: \; z = {\rm lim} \, \phi(x_n), \; x_n \to \Gamma_{0}\}$.
To show that
$n^1(z)=0$ for $z \in {\cal O}_0$ it is enough to observe that
$$
D \psi(z) ([1,0,0]) = [1+U^1(x),U^2(x),U^3(x)],
$$
where $x=\psi(z)$. But for $x \in \Gamma_0$ the vector on the r.h.s is tangent to $\Gamma_0$
since $U \cdot n|_{\Gamma_0} =0$.
We can conclude that on ${\cal O}_0$ the image in $\psi$ of a straight line
$\{(s,z_2,z_3): \, s \in (0,b)\}$ is a curve tangent to $\Gamma_0$, and thus
${\cal O}_0$ is a sum of such lines and so we have $n^1(z)=0$. The proof
of lemma \ref{lem_change} is completed. 
\hfill
$\square$

\noindent \emph{Proof of Lemma \ref{lem_S}.}
First we define $S(h)$ for a continuous function $h$ as
\begin{equation} \label{S}
S(h)(x) = w_{in}(0,\phi_2(x),\phi_3(x))
+ \int_{0}^{\phi_1(x)} h (\psi(s,\phi_2(x),\phi_3(x))) \,ds.
\end{equation}
The condition $n^1=0$ on $\phi(\Gamma_0)$ guarantees that a straight line
$(s,z_1,z_2): s \in (0,b)$ has a picture in $\Omega$ and thus we integrate along a curve contained in $\Omega$.
It means that $S$ is well defined for continuous functions defined on $\Omega$ and
the construction of $\psi$ clearly ensures that $S$ satisfies (\ref{def_S}). Next we have to extend $S$
on $L_2(\Omega)$. To this end we show the estimate (\ref{est_S}) for continuous $h$.
Let $\Omega_{x_1}$ denote an $x_1$-cut of $\Omega$ and let $\bar x := (x_2,x_3)$.
Then by (\ref{S}) we have
$$\begin{array}{c}
\|S(h)\|_{L_2(\Omega_{x_1})}^2 =
\int_{\Omega_{x_1}} \Big[ w_{in}(0,\phi_2(x),\phi_3(x)) + \int_0^{\phi_1(x)} h(\psi(s,\phi_2(x),\phi_3(x))) \,ds \Big]^2 \, d \bar x
\\
\leq 2 \|w_{in}\|_{L_2(\Gamma_{in})}^2 + C \int_{\Omega_{x_1}} \int_0^{\phi_1(x)} h^2(\psi(s,\phi_2(x),\phi_3(x))) \,ds \,d\bar x
\leq C \, \big[ \|w_{in}\|_{L_2(\Gamma_{in})}^2 + \|h\|_{L_2(\Omega)}^2 \big].
\end{array}
$$
The above estimate holds for every $x_1 \in (0,L)$ what implies (\ref{est_S}).
Now we can define $S(v)$ for $v \in L_2(\Omega)$ using a standard density argument,
the details are given in (\cite{TP1}).
To show the uniqueness of the solution we can rewrite the r.h.s of (\ref{def_S}) as
\begin{equation} \label{def_S_1a}
\left\{ \begin{array}{lcr}
\partial_{x_1} w + \frac{U^{2}}{1+U^{1}} \partial_{x_2} w + \frac{U^{3}}{1+U^{1}} \partial_{x_3} w = \frac{h}{1+U^{1}} & \textrm{in} & \Omega, \\
w = w_{in} & \textrm{on} & \Gamma_{in},
\end{array} \right.
\end{equation}
where the equation is to be understood in the sense of distribution, and, treating $x_1$ as a "time" variable, adapt Di Perna - Lions theory of transport equation (\cite{DPL})
that implies the uniqueness of solution to (\ref{def_S_1a}) in the class $L_{\infty}(L_2)$.
The proof is thus complete. \hfill $\square$
	
\section{Linearization and a priori bounds}
In this section we deal with the linear system 
 
\begin{equation} \label{system_lin}
\begin{array}{lcr}
\partial_{x_1} u -{\rm div}\, {\bf S}(\nabla u) +
p_1 \nabla \sigma + p_2 \nabla \eta =  F & \mbox{in} & \Omega,\\
{\rm div} \, u + \partial_{x_1}\sigma + U \cdot \nabla \sigma = G
& \mbox{in}& \Omega,\\
r_0\partial_{x_1}\eta + r_1 {\rm div}\,u  -\kappa \Delta \eta = H, & \mbox{in} & \Omega, \\
({\bf S}(\nabla u)n)\cdot  \tau_i +\alpha u \cdot \tau_i = B_i, \quad i=1,2
&\mbox{on} & \Gamma, \\
n\cdot  u = 0 & \mbox{on} & \Gamma,\\
\sigma = \sigma_{in} & \mbox{on} & \Gamma_{in},\\
\kappa \frac{\partial \eta}{\partial n} + L \eta =0 & \mbox{on} & \Gamma
\end{array}
\end{equation}
with given functions $(F,G,H,B_i,\sigma_{in}, U) \in (L_p \times W^1_p \times L_p \times W^{1-1/p}_p(\Gamma) \times W^1_p(\Gamma_{in}) \times W^2_p)$
such that 
\begin{equation} \label{3.1}
U \cdot n = 0 \mbox{ on } \Gamma_0, \quad U \cdot n >0 \mbox{ on } \Gamma_{out}, \quad U\cdot n <0 \mbox{ on } \Gamma_{in}.
\end{equation}
Note that this system actually represents \eqref{system}, where we only redefined a few constants.
The main difficulty lies in deriving appropriate estimates for (\ref{system_lin}).
We start with the energy estimates, then applying the properties of slip boundary conditions,
Helmholtz decomposition of the velocity and classical elliptic theory we derive the $L_p$ estimates.

Apart from solving the system (\ref{system_lin}), the estimates derived in this section are used later,
combined with (\ref{est_FGH}), to show the convergence of the sequence of approximations (\ref{system_seq}).

\subsection{Energy estimates}

\begin{lem} \label{l 8}
Let $(u,\sigma, \eta)$ solve system (\ref{system_lin}) with 
$(F,G,H,B_i,\sigma_{in},U) \in V^* \times L_2 \times L_2 \times L_2(\Gamma) \times L_2(\Gamma_{in}) \times W^2_p$ 
with $\|U\|_{W^2_p}$ small enough, $U \cdot n$ on $\Gamma$ fulfilling conditions \eqref{3.1} above, 
where $V^*$ is the dual space to
\begin{equation} \label{def_V}
V = \{v \in W^1_2(\Omega): v \cdot n|_{\Gamma} = 0 \}.
\end{equation}
Assume further that $\kappa,L$ and $\alpha$ fulfill the assumptions of Theorem \ref{main_thm}.
Then the following estimate holds: 
\begin{equation} \label{ene}
\begin{array}{c}
\|u\|_{W^1_2} + \|\sigma\|_{L_{\infty}(L_2)} + \|\eta\|_{W^1_2}  \\
\leq C \big[ \|F\|_{V^*} + \|G\|_{L_2} + \|H\|_{L_2} + \|B_i\|_{L_2(\Gamma)} + \|\sigma_{in}\|_{L_2(\Gamma_{in})}\big].
\end{array}
\end{equation}
\end{lem}  

\noindent
\begin{proof}
We multiply (\ref{system_lin})$_1$ by $u$ and integrate
applying the following identity:
\begin{equation} \label{basic_id}
\begin{array}{c}
\int_{\Omega} -{\rm div}\,{\bf S}(u) \, v \,dx
=\int_{\Omega} {\bf S}(u) : \nabla v   \,dx
- \int_{\Gamma} ({\bf S}(u) n) \cdot   v \,dS.
\end{array}
\end{equation}
With application of the Korn inequality and the boundary condition (\ref{system_lin})$_4$
we get
\begin{equation} \label{ene1}
\begin{array}{c}
C \|u\|_{W_1^2}^2 - \int_{\Omega} \eta {\rm div}\,u\, dx  
\leq \int_{\Omega} F \cdot u \,dx + \int_{\Omega} G\sigma \,dx \\
+ \frac{p_1}{2} \int_{\Omega} \sigma^2 {\rm div} \, U \,dx
+ \frac{p_1}{2} \int_{\Gamma_{in}} \sigma_{in}^2 (1- U\cdot n)\, dS + \sum_{i=1}^2\int_\Gamma B_i u \cdot \tau_i \, dS. 
\end{array}
 \end{equation}
Next we divide (\ref{system_lin})$_3$ by $r_1$ (recall it is constant), multiply by $\eta$ and integrate. 
It yields
\begin{displaymath}
\begin{array}{c}
\frac{\kappa}{r_1} \int_{\Omega} |\nabla \eta|^2 \,dx + \frac{L}{r_1} \int_{\Gamma_0} \eta^2 \,dS\\
- \frac{r_0}{2r_1} \int_{\Gamma_{in}} \eta^2 \,dS + \frac{r_0}{2r_1} \int_{\Gamma_{out}} \eta^2 \,dS
+ \int_{\Omega} \eta {\rm div}\,u \,dx = \int_{\Omega} H\eta \,dx,
\end{array}
\end{displaymath} 
hence
\begin{equation} \label{ene2}
\frac{\kappa}{r_1} \int_{\Omega} |\nabla \eta|^2 \,dx + \frac{L}{r_1} \int_{\Gamma_0} \eta^2 \,dS
+ \int_{\Omega} \eta {\rm div}\,u \,dx 
\leq \int_{\Omega} H\eta \,dx + \frac{r_0}{2r_1} \int_{\Gamma_{in}} \eta^2 \,dS.
\end{equation}
To get rid of the boundary term on the r.h.s we can apply the Poincar\'e inequality (\ref{Poin1})
and rewrite (\ref{ene2}) as
\begin{equation}
\frac{1}{r_1}\Big(\kappa-\frac{C_P}{2}\Big) \int_{\Omega} |\nabla \eta|^2 \,dx +
\frac{1}{r_1}\Big(L-\frac{C_P}{2}\Big) \int_{\Gamma_0} \eta^2 \,dS
+ \int_{\Omega} \eta {\rm div}\,u \,dx 
\leq  \int_{\Omega} H\eta \,dx,
\end{equation}
where $C_P$ is the constant from (\ref{Poin1}). 
For $\kappa$ large enough and $L$ large enough on $\Gamma_0$ the first two terms on the l.h.s. will be positive. 
Now we can combine (\ref{ene1}) and (\ref{ene2}) obtaining
\begin{displaymath} 
\begin{array}{c}
C(\Omega,\kappa,L) \big[ \|u\|_{W^1_2}^2 + \int_{\Omega} |\nabla \eta|^2 \,dx + \|\eta\|^2_{L_2(\Gamma_0)} \big] 
\leq 
\int_{\Omega} H\eta \,dx + \int_{\Omega} F \cdot u \,dx \\ + \int_{\Omega} G\sigma \,dx
+ \frac{1}{2} \int_{\Omega} \sigma^2 {\rm div} \, U \,dx
+ \frac{1}{2} \int_{\Gamma_{in}} \sigma_{in}^2 (1-U\cdot n) \,dS + \sum_{i=1}^2 \int_\Gamma B_i u \cdot \tau_i \, dS. 
\end{array}
\end{displaymath}
Using (\ref{Poin2}), the imbedding $W^1_p \hookrightarrow L_\infty$ and H\"older's inequality we derive
\begin{equation} \label{ene3}
\begin{array}{c}
\|u\|_{W^1_2}^2 + \|\eta\|_{W^1_2}^2 
\leq \|F\|_{V^*} \|u\|_{W^1_2} + \|G\|_{L_2} \|\sigma\|_{L_2} \\ + \|H\|_{L_2} \|\eta\|_{L_2}
+ C \|\sigma_{in}\|^2_{L_2(\Gamma_{in})} + E \|\sigma\|^2_{L_2} + \sum_{i=1}^2\|B_i\|_{L_2(\Gamma)}\|u\|_{L_2(\Gamma)}.
\end{array}
\end{equation}
Recall that E denotes a small constant and so to complete the proof we have to find the bound on
$\|\sigma\|_{L_\infty(L_2)}$. But this follows from the existence result in the previous subsection, see \eqref{est_S}. Hence we have
\begin{equation} \label{ene4}
\|\sigma\|_{L_\infty(L_2)} \leq C \Big( \|u\|_{W^1_2} + \|G\|_{L_2} + \|\sigma_{in}\|_{L_2(\Gamma_{in})}\Big) .
\end{equation}
Combining (\ref{ene3}) and (\ref{ene4}) we conclude (\ref{ene}). \end{proof}

\subsection{$L_p$ estimates}
  
The main result of this section is 
\begin{lem} \label{lem_est_lin}
Let $(u,\sigma,\eta)$ solve  system (\ref{system_lin}) with 
$(F,G,H,B_i,\sigma_{in}, U) \in L_p \times W^1_p \times L_p \times W^{1-1/p}_p(\Gamma) \times W^1_p(\Gamma_{in}) \times W^2_p$,
where $\|U\|_{W^2_p}$ is small enough and $\kappa, L, \alpha$ satisfy the assumptions of Theorem \ref{main_thm}. Then 
\begin{equation} \label{est_lin}
\begin{array}{c}
\|u\|_{W^2_p} + \|\sigma\|_{W^1_p} + \|\eta\|_{W^2_p} \\
\leq C \, \big[\|F\|_{L_p} + \|G\|_{W^1_p} + \|H\|_{L_p} + \|\sigma_{in}\|_{W^1_p(\Gamma_{in})}
+ \sum_{i=1}^2\|B_i\|_{W^{1-1/p}(\Gamma)} \big] := D_{lin}.
\end{array}
\end{equation}
\end{lem}
\noindent
The proof of (\ref{est_lin}) will be performed in several consecutive lemmas.
We show the bound on $\|\eta\|_{W^2_p}$. 
Then we proceed with the bound on $\|\sigma\|_{W^1_p}$ which is the most demanding.
First we show bound on the vorticity which, together with Helmholtz decomposition
of the velocity, makes possible to eliminate the term ${\rm div}\,u$ from the continuity equation,
leading to (\ref{trans}). Using this equation we show the bound on the density.
Then (\ref{est_lin}) easily follows from the classical elliptic estimate on $\|u\|_{W^2_p}$.
 
We start with the bound on $\eta$; to show it let us
rewrite (\ref{system_lin})$_{3,7}$ as 
\begin{equation}
\begin{array}{lr}
- \kappa \Delta \eta = H - r_0 \partial_{x_1}\eta - r_1 {\rm div}\,u & \mbox{ in } \Omega,\\
\frac{\partial \eta}{\partial n}|_{\Gamma} = - \frac{L}{\kappa} \eta & \mbox{ on } \Gamma.
\end{array}
\end{equation}
The classical elliptic estimate (\cite{ADN1}, \cite{ADN2}) for the above system yields
\begin{displaymath}
\|\eta\|_{W^2_p} \leq C \big[ \|H\|_{L_p} + \|\eta\|_{W^1_p} + \|\eta\|_{W^{1-1/p}_p(\Gamma)} + \|u\|_{W^1_p} \big].
\end{displaymath}
Applying the trace theorem to the boundary term and then the interpolation inequality (\ref{int1}) we get
\begin{equation} 
\|\eta\|_{W^2_p} \leq C \big[\|H\|_{L_p} + \|\eta\|_{W^1_2} + \|u\|_{W^1_2}\big] + \epsilon \|u\|_{W^2_p}.
\end{equation}
By the energy estimate (\ref{ene}) we conclude
\begin{equation} \label{eta_w2p}
\|\eta\|_{W^2_p} \leq C D_{lin} + \epsilon \|u\|_{W^2_p}.
\end{equation}
Now we proceed towards the bound on $\|\sigma\|_{W^1_p}$. As mentioned before, we start with the estimate
on the vorticity. Taking the curl of (\ref{system_lin})$_1$ we get
\begin{equation}   \label{system_rot}
\begin{array}{lcr}
- \mu \Delta \alpha = {\rm curl}\,[F-\partial_{x_1}u] & \mbox{in} & \Omega, \\
\alpha \cdot \tau_2 = (2 \chi_1 - \frac{\alpha}{\mu}) u \cdot \tau_1 + \frac{B_1}{\mu} & \mbox{on} & \Gamma, \\
\alpha \cdot \tau_1 = (\frac{\alpha}{\mu} - 2 \chi_2) u \cdot \tau_2 - \frac{B_2}{\mu} & \mbox{on} & \Gamma, \\
{\rm div}\,\alpha = 0 & \mbox{on} & \Gamma,
\end{array}
\end{equation}
where $\chi_i$ denote the curvatures of the curves generated by tangent vectors $\tau_i$.
In order to show the boundary relations (\ref{system_rot})$_{2,3}$ it is enough to differentiate
(\ref{system_lin})$_4$ with respect to the tangential directions and apply (\ref{system_lin})$_3$.
A rigorous proof is given in \cite{MR} or \cite{TP1}. 

The condition ${\rm div}\, \alpha = 0$ in $\Omega$ results simply
from the fact that $\alpha = {\rm curl}\,u$. We introduce this relation as a boundary condition
(\ref{system_rot})$_4$, that completes
the conditions on the tangential parts of the vorticity. Note that
the boundary conditions (\ref{system_rot})$_{2,3}$ give the tangential parts of
the vorticity on the boundary the regularity of the velocity itself and the data.
We will now use this feature of slip boundary conditions to show the higher estimate on the vorticity
(see \cite{PM1},\cite{PM2}, \cite{PMMP2}).

%
%
%
%
For the above system we have (see \cite{Z}, Theorem 10.4):
\begin{equation} \label{alpha_w1p_1}
\|\alpha\|_{W^1_p} \leq C \, \big[\|F\|_{L_p} + \|\partial_{x_1}u\|_{L_p} + \sum_{i=1}^2\|B_i\|_{W^{1-1/p}_p(\Gamma)} + \|u\|_{W^{1-1/p}_p(\Gamma)} \big].
\end{equation}
Applying the interpolation inequality (\ref{int1}) and then (\ref{ene}) 
we arrive at
\begin{equation}  \label{rotuw1p}
\|\alpha\|_{W^1_p} \leq C(\epsilon) D_{lin}  + \epsilon \|u\|_{W^2_p}
\end{equation}
for any $\epsilon>0$, where $D_{lin}$ is defined in (\ref{est_lin}). 
Now we introduce the Helmholtz decomposition of the velocity (see e.g. \cite{Ga})
\begin{equation} \label{Helm}
u = \nabla \phi + A,
\end{equation}
where $\frac{\partial \phi}{\partial n}|_{\Gamma}=0$ and ${\rm div}\, A =0$.
We see that the field $A$ satisfies the following system
\begin{equation}
\begin{array}{lcr}
{\rm curl}\, A = \alpha & \mbox{in} & \Omega, \\
{\rm div}\, A = 0 & \mbox{in} & \Omega, \\
A \cdot n = 0 & \mbox{on} & \Gamma.
\end{array}
\end{equation}
For this system we have (see \cite{Sol}):
$
\|A\|_{W^2_p} \leq C \, \|\alpha\|_{W^1_p} ,
$
what by (\ref{alpha_w1p_1}) can be rewritten as
\begin{equation}  \label{Aw2p}
\|A\|_{W^2_p} \leq C(\epsilon) \, \big[ \|F\|_{L_p}   + \sum_{i=1}^2\|B_i\|_{W^{1-1/p}_p(\Gamma)}
	 \big] + \epsilon \|u\|_{W^2_p}
\end{equation}
for any $\epsilon>0$.
Now we substitute the Helmholtz decomposition to (\ref{system_lin})$_1$. We get
\begin{equation}  \label{nablaK}
\textstyle \nabla [ -(\lambda + \frac 43 \mu) \Delta \phi + p_1 \, \sigma ] =
F - \partial_{x_1} A + \mu \Delta A  - \partial_{x_1} \nabla \phi - p_2\nabla \eta,
\end{equation}
but $\Delta \phi = {\rm div}\, u$.  We denote
\begin{equation}
\textstyle -(\lambda + \frac 43 \mu) {\rm div}\,u + p_1 \, \sigma = \bar K.
\end{equation}
Combining the last equation with (\ref{system_lin})$_2$ we arrive at
\begin{equation} \label{trans}
 \gamma \sigma + \partial_{x_1}\sigma + U\cdot  \nabla \sigma = K,
\end{equation}
where $ \gamma = \frac{p_1}{\lambda + \frac 43 \mu}$ and
\begin{equation} \label{K}
K = \frac{\bar K}{\lambda + \frac 43 \mu}+G.
\end{equation}
Equation (\ref{trans}) makes  possible to estimate the $W^1_p$-norm of the density
in terms of $W^1_p$ - norm of $K$, which in turn
will be controlled by (\ref{nablaK}) using interpolation and the energy estimate (\ref{ene1}).
First we estimate $\|\sigma\|_{W^1_p}$ in terms of $K$. The result is stated in the following lemma
\begin{lem}
Assume that $\sigma$ satisfies  equation (\ref{trans}) with $K \in W^1_p$. Then
\begin{equation} \label{w_w1p}
\|\sigma\|_{W^1_p} \leq C \, \big[ \|K\|_{W^1_p} + \|\sigma_{in}\|_{W^1_p(\Gamma_{in})} \big].
\end{equation}
\end{lem}
\begin{proof} 
In order to find a bound on $\|\sigma\|_{L_p}$ we multiply (\ref{trans}) by $|\sigma|^{p-2} \sigma$
and integrate over $\Omega$. Integrating by parts
and next using the boundary conditions we get

\begin{equation} \label{w1}
\begin{array}{c}
 \gamma \|\sigma\|_{L_p}^p - \frac{1}{p} \int_{\Omega} {\rm div}\, U \, |\sigma|^p \,dx + \frac{1}{p} \int_{\Gamma_{out}} |\sigma|^p \, (1+U \cdot n)  \,dS \\
\leq \|K\|_{L_p} \, \|\sigma\|_{L_p}^{p-1} + \frac{1}{p} \int_{\Gamma_{in}} |\sigma_{in}|^p \, (1-U \cdot n) \,dS.
\end{array}
\end{equation}
With the smallness of $\|U \|_{W^2_p}$, the above implies 
$$
C \, \|\sigma\|_{L_p}^p \leq
\|K\|_{L_p} \, \|\sigma\|_{L_p}^{p-1} + C \, \|\sigma_{in}\|_{L_p(\Gamma_{in})}^p,
$$
and so
\begin{equation} \label{w}
\|\sigma\|_{L_p} \leq C \, \big[ \|K\|_{L_p} + \|\sigma_{in}\|_{L_p(\Gamma_{in})} \big] .
\end{equation}

Now we estimate the derivatives. In order to find a bound on $\partial_{x_i}\sigma$
we differentiate (\ref{trans}) with respect to $x_i$. Note that $\nabla^2 \sigma$ is not defined in general, however, for
%
\begin{equation} \label{tilde-u}
\tilde u := [1 + U^{1}, U^{2}, U^{3}]
\end{equation}
and $\sigma \in W^1_p$ we may set
 $\tilde u \cdot \nabla \partial_{x_i}\sigma := \partial_{x_i}(\tilde u \cdot \nabla \sigma) -  \partial_{x_i}\tilde u \cdot \nabla \sigma \in L_p$.

Hence we can differentiate (\ref{trans}) with respect to $x_i$, multiply by $|\partial_{x_i}\sigma|^{p-2} \partial_{x_i}\sigma$ and integrate.
Since $\partial_{x_i}\tilde u = \partial_{x_i} U$, we have
$$
\Big|\int_{\Omega} \partial_{x_i}\tilde u \cdot (|\partial_{x_i} \sigma|^{p-2} \partial_{x_i}\sigma \nabla \sigma) \,dx\Big| \leq
\|\nabla U\|_{L_{\infty}} \, \|\nabla \sigma\|_{L_p}^p \leq C \, \|U\|_{W^2_p} \, \|\nabla \sigma\|^p_{L_p}.
$$
Next, since $\tilde u \cdot \nabla \partial_{x_i}\sigma \in L_p$, we can write   
$$
\begin{array}{c}
\int_{\Omega} \tilde u \cdot |\partial_{x_i}\sigma|^{p-2} \partial_{x_i}\sigma \nabla \partial_{x_i}\sigma \,dx =
\frac{1}{p} \int_{\Omega} \tilde u \cdot \nabla |\partial_{x_i}\sigma|^p \,dx \\[5pt]
= - \frac{1}{p} \int_{\Omega} |\partial_{x_i}\sigma|^p \, {\rm div}\, \tilde u \,dx + \frac{1}{p} \int_{\Gamma} |\partial_{x_i}\sigma|^p \, \tilde u \cdot n \, dS 
\\[5pt]
= - \frac{1}{p} \int_{\Omega} |\partial_{x_i}\sigma|^p \, {\rm div}\, U \,dx - \frac{1}{p} \int_{\Gamma_{in}} |\partial_{x_i}\sigma_{in}|^p \, (1+U^1) \,dS
+ \frac{1}{p} \int_{\Gamma_{out}} |\partial_{x_i}\sigma|^p \, (1+U^1) \,dS.
\end{array}
$$
For $i=2,3$ we have $\sigma_{in,x_i} \in L_p(\Gamma_{in})$ and hence the above defines the trace of $|\partial_{x_i}\sigma|^p$ on $\Gamma_{out}$.
We arrive at
\begin{equation}  \label{wxi}
\begin{array}{c}
 \gamma \|\partial_{x_i}\sigma\|_{L_p}^p - \frac{1}{p} \int_{\Omega} {\rm div} \, U \, |\partial_{x_i}\sigma|^p \,dx + \frac{1}{p} \int_{\Gamma_{out}} |\partial_{x_i}\sigma|^p \, (1+U^1)  \,dS  \\
\leq \|\partial_{x_i}K\|_{L_p} \, \|\partial_{x_i}\sigma\|_{L_p}^{p-1} + \frac{1}{p} \int_{\Gamma_{in}} |\partial_{x_i}\sigma_{in}|^p \, (1+U^1) \,dS + C \, \|U\|_{W^2_p} \|\nabla \sigma\|_{L_p}^p.
\end{array}
\end{equation}
For $i=2,3$ (\ref{wxi}) gives straightforward bound on $\|\partial_{x_i}\sigma\|_{L_p}$.
In order to estimate $\partial_{x_1}\sigma$ we also differentiate (\ref{trans}) with respect to $x_1$
and multiply by $|\partial_{x_1}\sigma|^{p-2} \partial_{x_1}\sigma$. The difference is that
$\partial_{x_1}\sigma$ is not given on $\Gamma_{in}$. To overcome this difficulty we can observe that
on $\Gamma_{in}$  equation (\ref{trans}) reduces to
$$
\gamma \sigma_{in} + U^2 \, \partial_{x_2}\sigma_{in} + U^3 \, \partial_{x_3}\sigma_{in} + [1 + U^1] \,\partial_{x_1}\sigma = K,
$$
what can be rewritten as
$$
\partial_{x_1}\sigma = \frac{1}{1+U^1} \, \big[ K -  \gamma \sigma_{in} - U_{\tau} \cdot \nabla_{\tau} \sigma_{in} \big],
$$
where the lower index $\tau$ denotes the tangential component. Thus we have
$$
\|\partial_{x_1}\sigma\|_{L_p(\Gamma_{in})} \leq C \, \big[ \|K|_{\Gamma_{in}}\|_{L_p(\Gamma_{in})} + \|\sigma_{in}\|_{W^1_p(\Gamma_{in})} \big].
$$
Using this bound in (\ref{wxi}), $i=1$, we arrive at the estimate
\begin{equation} \label{wx1}
\|\partial_{x_1}\sigma\|_{L_p}^p \leq C \, \big[ \|\partial_{x_1}K\|_{L_p} \, \|\partial_{x_1}\sigma\|_{L_p}^{p-1} + \|U\|_{W^2_p} \, \|\nabla \sigma\|_{L_p}^p
+ \|K\|_{L_p(\Gamma_{in})}^p + \|\sigma_{in}\|_{W^1_p(\Gamma_{in})}^p \big].
\end{equation}
Applying the trace theorem to the term $\|K\|_{L_p(\Gamma_{in})}$ and then
combining (\ref{wxi}) (for $x_2$ and $x_3$) with (\ref{wx1}) we get
\begin{equation}
\|\nabla \sigma\|_{L_p}^p \leq C \, \big[ \|\nabla K\|_{L_p} \|\nabla \sigma\|_{L_p}^{p-1}
+ \|U\|_{W^2_p} \|\nabla \sigma\|_{L_p}^p + \|K\|_{W^1_p}^p + \|\sigma_{in}\|_{W^1_p(\Gamma_{in})}^p \big].
\end{equation}
The term $\|U\|_{W^2_p} \|\nabla \sigma\|_{L_p}^p$ can be put on the l.h.s. due to the smallness assumption
and thus by Young's inequality 
\begin{equation}
\|\nabla \sigma\|_{L_p} \leq C \, \big[ \|K\|_{W^1_p} + \|\sigma_{in}\|_{W^1_p(\Gamma_{in})} \big],
\end{equation}
what combined with (\ref{w}) yields
\begin{equation} \label{w_w1p_1}
\|\sigma\|_{W^1_p} \leq C \, \big[ \|K\|_{W^1_p}  + \|\sigma_{in}\|_{W^1_p(\Gamma_{in})} \big].
\end{equation}
The lemma is proved. \end{proof}

Now we estimate $K$ in terms of the data. The result is 
\begin{lem}   \label{lemK}
Let $K$ be defined in (\ref{K}). Then $\forall \delta>0$ we have
\begin{equation}    \label{lemK_teza}
\|K\|_{W^1_p} \leq
\delta \|u\|_{W^2_p} + C(\delta) D_{lin},
\end{equation}
where $D_{lin}$ is defined in (\ref{est_lin}).
\end{lem}
\noindent
\begin{proof}
Applying first (\ref{int1}) and then (\ref{ene1})
we get
\begin{equation}     \label{Klp}
\|K\|_{L_p} \leq \delta_1 \|\nabla K\|_{L_p}
+ C(\delta_1) \, \big[ \|F\|_{L_2} + \|G\|_{L_2} + \|B\|_{L_2(\Gamma)} \big].
\end{equation}
Hence, it is enough to find the bound on $\|\nabla K\|_{L_p}$. By (\ref{nablaK}) we have
$$
\|\nabla K\|_{L_p} \leq C \, \big[ \|F\|_{L_p} + \|G\|_{W^1_p} + \|A\|_{W^2_p} + \|\partial_{x_1} \nabla\phi\|_{L_p} 
+ \|\nabla \eta\|_{L_p}\big],
$$
where $u = \nabla \phi + A$ is the Helmholtz decomposition of the velocity.
Applying (\ref{Aw2p}), (\ref{eta_w2p}), trace theorem and (\ref{int1})
we arrive at
\begin{equation}    \label{nablahlp}
\begin{array}{c}
\|\nabla K\|_{L_p} \leq C \, \big[ \|F\|_{L_p} + \|G\|_{W^1_p} +\|H\|_{L^p} + \|B\|_{W^{1-1/p}_p(\Gamma)} + \|\sigma_{in}\|_{W^1_p(\Gamma_{in})} \big] \\
+ \delta_1 \|u\|_{W^2_p} + C(\delta_1) \big[\|F\|_{L_2} + \|G\|_{L_2} + \|B\|_{L_2(\Gamma)} \big].
\end{array}
\end{equation}
Combining (\ref{Klp}) and (\ref{nablahlp}) we get (\ref{lemK_teza}).
\end{proof}

\noindent
The bound on $\sigma$ now follows directly.
Substituting (\ref{lemK_teza}) to (\ref{w_w1p}) we get
\begin{equation} \label{w_w1p_ost}
\|\sigma\|_{W^1_p} \leq \delta \|u\|_{W^2_p} + C(\delta) D_{lin},
\end{equation}
where $D_{lin}$ is given by (\ref{est_lin}).
Now the only missing piece to complete the proof of (\ref{est_lin}) is the bound on $\|u\|_{W^2_p}$.
To show it note that the velocity satisfies the Lam\'e system
\begin{eqnarray}
\begin{array}{lcr}
u_{x_1} -{\rm div}\, {\bf S}(\nabla u) =  F - p_1 \nabla  \sigma - p_2 \nabla \eta & \mbox{in} & \Omega,\\
({\bf S}(\nabla u) n)\cdot \tau_i +\alpha \ u \cdot \tau_i = B_i, \quad i=1,2
&\mbox{on} & \Gamma, \\
n\cdot  u = 0 & \mbox{on} & \Gamma.\\
\end{array}
\end{eqnarray}
The classical theory of elliptic equations (\cite{ADN1},\cite{ADN2}) yields
$$
\|u\|_{W^2_p} \leq C \, \big[D_{lin} + \|u\|_{W^1_p}\big].
$$
Applying (\ref{int1}) to the term $\|u\|_{W^1_p}$ and then (\ref{ene1})
we get
\begin{equation} \label{u_w2p}
\|u\|_{W^2_p} \leq C \,  D_{lin}.
\end{equation}
\emph{Proof of Lemma \ref{lem_est_lin}.} We combine (\ref{eta_w2p}), (\ref{u_w2p}) and (\ref{w_w1p_ost}). 
In (\ref{w_w1p_ost}) we choose for example 
$\delta = \frac{1}{2C}$ where $C$ is the constant from (\ref{u_w2p}). Now (\ref{est_lin}) follows immediately. \hfill $\square$

\section{Solution of the linear system}
In this section we solve the linear system (\ref{system_lin}) and hence show
that the sequence (\ref{system_seq}) is well defined. We start with showing the existence of an appropriately defined weak
solution applying the Galerkin method modified in a way to deal with the continuity equation.
Then we show the regularity of the solutions for the given  regularity of the data.  
For simplicity let us denote 
\begin{equation} \label{tilde_u}
\tilde u := [1 + U^1, U^2, U^3].
\end{equation}

\subsection{Weak solution}
In order to define a weak solution to (\ref{system_lin}) we recall the definition of the space $V$
(\ref{def_V}). 
%
%
A natural definition of a weak solution to the system (\ref{system_lin}) is 
$(u,\sigma,\eta) \in V \times L_\infty(L_2) \times W^1_2$ such that

\noindent
{\bf 1.} the identity
\begin{eqnarray} \label{weak1}
\int_{\Omega} \{ v \cdot \partial_{x_1} u  + {\bf S}(\nabla u) :\nabla v
- p_1 \sigma \, {\rm div}\,v - p_2 \eta {\rm div}\,v \} \,dx
+ \sum_{i=1}^2 \int_{\Gamma} \alpha (u \cdot \tau_i) \, (v \cdot \tau_i) \,dS  \nonumber\\
= \int_{\Omega} F \cdot v \,dx + \sum_{i=1}^2 \int_{\Gamma} B_i (v \cdot \tau_i) \,dS
\end{eqnarray}
is satisfied $\forall \; v \in V$; 

\noindent
{\bf 2.} (\ref{system_lin})$_2$ is satisfied in the following sense:
\begin{equation} \label{weak2}
-\int_{\Omega} \sigma \tilde u \cdot \nabla \phi \,dx - \int_{\Omega} \sigma \phi \, {\rm div}\, \tilde u \,dx =
\int_{\Omega} \phi (G - {\rm div}\, u) \,dx + \int_{\Gamma_{in}} \sigma_{in} \phi \, dS,
\end{equation}
$\forall \; \phi \in  C^{1}(\overline \Omega): \phi|_{\Gamma_{out}}=0$,
where $\tilde u$ is defined in (\ref{tilde_u}), and

\noindent
{\bf 3.} the equation
\begin{equation} \label{weak3}
r_0\int_{\Omega} \partial_{x_1}  \eta w + r_1 \int_{\Omega} w {\rm div}\,u + \kappa \int_{\Omega} \nabla w \cdot \nabla \eta 
+ \int_{\Gamma} Lw\eta \,dS = \int_{\Omega} Hw\,dx
\end{equation}
holds for all $w \in W^1_2$. 

We introduce an orthonormal basis of $V$ formed by $\{\omega_i\}_{i=1}^{\infty}$. We consider finite dimensional spaces:
\mbox{$V^N = \{ \sum_{i=1}^N \alpha_i \omega_i: \; \alpha_i \in \mathbb{R} \} \subset V$}.
The sequence of
approximations to the velocity will be searched for as
$
u^N = \sum_{i=1}^N c_i^N \, \omega_i.
$
Due to the equation (\ref{system_lin})$_2$ we have to define the approximations to the density
in an appropriate way. Namely, we set $\sigma^N = S(G - {\rm div}\, u^N)$,
where $S:L_2(\Omega) \to L_{\infty}(L_2)$ is defined in (\ref{def_S}), with $w_{in} = \sigma_{in}$. 

Finally, to define the approximations of the temperature we introduce $R:H^1 \to H^1$ as a weak solution
operator to equation (\ref{system_lin})$_{3,7}$, i.e. we set 
$$
\eta^N = R(u^N) \iff (\ref{weak3}) \; \textrm{holds with} \; \eta:= \eta^N \, \forall \; w \in W^1_2.
$$
The well-posedness of the operator $R$ is direct as it is just a solution operator for a standard elliptic equation.
To show it is well defined it is enough to recall the estimate from the proof of (\ref{ene}).
In the same way we show the estimate for (\ref{weak3}) which gives existence of $\eta$ satisfying
(\ref{weak3}) for given $u \in W^1_2$, hence $R$ is well defined.
In particular, we have
\begin{equation} \label{est_R}
\|R(u)\|_{W^1_2} \leq C [\|F\|_{L_2} + \|u\|_{W^1_2}].
\end{equation}

With the operators $S$ and $R$ well defined we can proceed with the Galerkin method.
Taking  $u = u^N = \sum_i c_i^N \, \omega_i$, $v = \omega_k$, $k=1 \ldots N$ and $\sigma = \sigma^N = S(G - {\rm div}\, u^N)$ in (\ref{weak1}), $\eta = \eta^N = R(u_N)$,
we arrive at a system of $N$ equations
\begin{equation}  \label{system_aprox}
B^N(u^N,\omega_k) = 0, \quad k = 1 \ldots N,
\end{equation}
where $B^N:V^N \to V^N$ is defined as
\begin{equation}
\begin{array}{c}
B^N(\xi^N,v^N) = \int_{\Omega} \big\{ v^N \partial_{x_1}\xi^N + {\bf S}(\xi^N):\nabla v^N \big\}\,dx \\
- p_1 \int_{\Omega} S(G-{\rm div}\,\xi^N) \, {\rm div}\,v^N \,dx 
- p_2\int_{\Omega} R(\xi^N) {\rm div}\,v^N \,dx\\
+ \sum_{j=1}^2\int_{\Gamma} [\alpha \, (\xi^N \cdot \tau_j) - B_j] \, (v^N \cdot \tau_j) \,dS - \int_{\Omega} F \cdot v^N \,dx.
\end{array}
\end{equation}
Now, if $u^N$ satisfies (\ref{system_aprox}) for $k = 1 \ldots N$ and  $\sigma^N$  
and $\eta^N$, defined as above, 
then the triple
$(u^N,\sigma^N,\eta^N)$ satisfies (\ref{weak1})--(\ref{weak3}) for $(v,\phi,w) \in (V^N \times  C^{1}(\overline\Omega)) \times W^1_2$,
$\phi|_{\Gamma_{out}}=0$.
We will call such a triple an approximate solution to (\ref{weak1})--(\ref{weak3}).

The following lemma gives existence of a solution to  system (\ref{system_aprox}):
\begin{lem}
Let $F,G \in L^2(\Omega)$, $\sigma_{in} \in L_2(\Gamma_{in})$, $B_i \in L_2(\Gamma)$, $i=1,2$.
Assume that $\|U\|_{W^2_p}$ and $\kappa$, $L$, $\alpha$ fulfill the assumptions of Theorem \ref{main_thm}.
Then there exists $u^N \in V^N$ satisfying (\ref{system_aprox}) for $k=1 \ldots N$.
Moreover,
\begin{equation}  \label{est_uN}
\|u^N\|_{W^1_2} \leq C(DATA).
\end{equation}
\end{lem}
\begin{proof}
We will apply a well-known tool of the Galerkin method, Lemma \ref{lem_P}. 
Thus we define $P^N:V^N \to V^N$ as
\begin{equation}   \label{P}
P^N(\xi^N) = \sum_k B^N(\xi^N,\omega_k) \omega_k \quad \textrm{for} \quad \xi^N \in V^N.
\end{equation}
We have to show that $(P^N(\xi^N),\xi^N)>0$ on some sphere in $V^N$. 
As $B$ is linear in the second variable, we have
\begin{equation}
\begin{array}{c}
\big( P(\xi^N),\xi^N \big) = B^N(\xi^N,\xi^N) =
\underbrace{\int_{\Omega} {\bf S}(\nabla \xi^N):\nabla \xi^N \,dx}_{I_1} \\
+ \underbrace{\int_{\Omega} \xi^N \partial_{x_1}\xi^N \,dx + \int_{\Gamma} \alpha (\xi^N \cdot \tau_i)^2 \,dS}_{I_2}
\underbrace{-p_1 \int_{\Omega} S(G - {\rm div}\, \xi^N) \, {\rm div}\,\xi^N  \,dx}_{I_3} \\
\underbrace{- p_2\int_{\Omega} R(\xi^N) {\rm div}\,v^N \,dS}_{I_4}
- \displaystyle \int_{\Omega} F \cdot \xi^N \,dx - \sum_{i=1}^2 \int_{\Gamma} B_i \, (\xi^N \cdot \tau_i) \,dS .
\end{array}
\end{equation}
By the Korn inequality we have 
\begin{equation} \label{est_P_1}
I_1+I_2 \geq C \, \|\xi^N\|_{W^1_2}^2
\end{equation}
for $\alpha$ large enough. To deal with $I_4$ 
let us denote for a moment $\eta^N = R(\xi^N)$. 
Then setting $w = \eta^N$ in (\ref{weak3}) (with $\xi^N$ instead of $u$) we have  
$$
I_4 = -p_2\int \eta^N {\rm div}\,\xi^N = \frac{p_2}{r_1} 
\Big[ \int_{\Omega} \partial_{x_1}\eta^N \eta^N\,dx  + \kappa \int_{\Omega}|\nabla \eta^N|^2 \,dx 
+ \int_{\Gamma} L (\eta^N)^2 \,dS - \int_\Omega  H \eta^N\,dx \Big]
$$  
and repeating the reasoning from the proof of (\ref{ene}) we infer
\begin{equation} \label{est_P_2}
I_4 \geq - \|H^N\|_{L_2} \|\eta^N\|_{L^2} \geq
- \|H^N\|_{L_2} \big(\|H^N\|_{L_2}+\|\xi^N\|_{W^1_2}\big).
\end{equation}
\noindent
We have to find a bound on $I_3$.
Denoting $\sigma^N = S (G - {\rm div}\, \xi^N)$ we have
\begin{equation} \label{I3_1}
-\int_{\Omega} \sigma^N \, {\rm div}\,\xi^N \,dx =
\int_{\Omega} \sigma^N ( \partial_{x_1} \sigma^N + U \cdot \nabla \sigma^N )\,dx
-\int_{\Omega} \sigma^N \, G \,dx.
\end{equation}
Using (\ref{est_S}) we get
\begin{equation}
-\int_{\Omega} \eta^N \, G^N \,dx \geq - \|\eta^N\|_{L^2} \, \|G^N\|_{L^2}
\geq - C \, \|G^N\|_{L_2} \, \big(\|G^N\|_{L_2} + \|\xi^N\|_{W^1_2} + \|\sigma_{in}\|_{L_2(\Gamma_{in})}\big).
\end{equation}
The remaining part of (\ref{I3_1}) is also not very difficult.
With the first integral on the r.h.s we have
\begin{equation} \label{I3_2}
\begin{array}{c}
\int_{\Omega} \sigma^N ( \partial_{x_1} \sigma^N + U \cdot \nabla \sigma^N )\,dx
= \frac 12 \int_\Omega \big(\partial_{x_1} |\sigma^N|^2 + U \cdot \nabla |\sigma^N|^2\big) \, dx \\[2pt]
= -\frac12 \int_{\Gamma_{in}} |\sigma^N|^2 \, dS + \frac 12 \int_{\Gamma_{out}} |\sigma^N|^2 \, dS + \frac 12 \int_{\Gamma} U\cdot n |\sigma^N|^2 \, dS - \frac 12 \int_{\Omega} |\sigma^N|^2 {\rm div}\, U \, dx \\[2pt]
\geq - \frac 12 \int_{\Gamma_{in}} \sigma_{in}^2 \, dS + \frac 12 \int_{\Gamma} U\cdot n |\sigma_{in}|^2 \, dS - \frac 12 \|{\rm div}\,  U\|_{L_\infty} \|\sigma^N\|_{L_2}^2 \\[2pt]
\geq -C \|\sigma_{in}\|_{L_2}^2 -E \|\sigma^N\|_{L_2}^2 \geq -C -E \big(\|G\|_{L_2}^2 + \|\sigma^N\|_{W^1_2}^2 + \|\sigma_{in}\|_{L_2}^2\big).
\end{array}
\end{equation}
%
Hence we have
\begin{equation} \label{est_P_3}
I_3 \geq - C \, \big(\|G\|^2_{L_2} +\|\sigma_{in}\|^2_{L_2(\Gamma_{in})}\big) -E \|\xi^N\|_{W^1_2}^2. 
\end{equation}
Combining (\ref{est_P_1}),(\ref{est_P_2}) and (\ref{est_P_3}) we conclude
\begin{equation}
\big( P^N(\xi^N),\xi^N \big) \geq C \,\big[ \|\xi^N\|_{W^1_2}^2 - D \, \|\xi^N\|_{W^1_2} - D^2 \big],
\end{equation}
where $D = \|F\|_{L^2(\Omega)}+\|G\|_{L^2(\Omega)} + \|H\|_{L^2(\Omega)} + \|\sigma_{in}\|_{L_2(\Gamma_{in})} + \sum_{i=1}^2\|B_i\|_{L_2(\Gamma)} $.
Thus there exists $M = M(\mu,\Omega,D)$ such that
$\big( P^N(\xi^N),\xi^N \big) > 0 \quad \textrm{for} \quad \|\xi^N\|_{W^1_2}=M$,
and applying Lemma \ref{lem_P} we conclude that
$\exists \xi^{N*}: \quad P^N(\xi^{N*})=0 \quad \textrm{and} \quad \|\xi^{N*}\|_{W^1_2} \leq M$.
By the definition of $P^N$, $u^N = \xi^{N*}$ is a solution to (\ref{system_aprox}).
\end{proof}

\noindent
Now showing the existence of weak solution is straightforward. The result is

\begin{lem} \label{lem_weak}
Assume that $F,G,H \in L_2(\Omega)$, $\sigma_{in} \in L_2(\Gamma_{in})$, $B_i \in L_2(\Gamma)$, $i=1,2$. 
Let $\|U\|_{W^2_p}$ be small enough and assume $\kappa,L,f$ satisfy the assumptions of Theorem \ref{main_thm}.
Then there exists $(u,\sigma,\eta) \in V \times L_\infty(L_2) \times W^1_2$, 
that is a weak solution to system (\ref{system_lin}).
Moreover, the weak solution satisfies  estimate (\ref{ene}).
\end{lem}
\noindent
\begin{proof} 
Let us set $\sigma^N=S(G-{\rm div}\,u^N)$ and $\eta^N = R(u^N)$, where $u^N$ is the solution
to (\ref{system_aprox}).
Estimates (\ref{est_S}), (\ref{est_R}) and (\ref{est_uN}) imply that
$\|u^N\|_{H^1}+\|\sigma^N\|_{L_{\infty}(L_2)} + \|\eta^N\|_{H^1}  \leq C(D_{lin})$. Thus, at least for a chosen subsequence (denoted however in the same way) 
$$
u^N \rightharpoonup u \quad {\rm in} \quad H^1, \qquad \sigma^N \rightharpoonup^* \sigma \quad {\rm in} \quad L_{\infty}(L_2)
\qquad \textrm{and} \quad \eta^N \rightharpoonup \eta \quad {\rm in} \; W^1_2
$$
for some $(u,\sigma,\eta) \in H^1 \times L_{\infty}(L_2) \times H^1$.
Passing to the limit in (\ref{weak1})--(\ref{weak3}) for $(u^N,\sigma^N,\eta^N)$
we conclude that $(u,\sigma,\eta)$ satisfies (\ref{weak1})--(\ref{weak3}), thus we have the weak solution.
To show the boundary condition on the density we can rewrite the r.h.s of (\ref{system_lin})$_{2,6}$ as
\begin{equation} \label{def_S_1}
\left\{ \begin{array}{lcr}
\partial_{x_1}\sigma + \frac{U^2}{1+U^1} \partial_{x_2}\sigma + \frac{U^3}{1+U^1} \partial_{x_3}\sigma = \frac{G-{\rm div}\, u} {1+U^1} & \textrm{in} & {\cal D'}(\Omega), \\
\sigma = \sigma_{in} & \textrm{on} & \Gamma_{in},
\end{array} \right.
\end{equation}
and, treating $x_1$ as a "time" variable, adapt Di Perna-Lions theory of transport equation (\cite{DPL})
that implies the uniqueness of solution to (\ref{def_S_1}) in the class $L_{\infty}(L_2)$.
The proof is thus complete.
\end{proof}

\subsection{Strong solution}
With the estimate \eqref{est_lin} the only problem is to tackle the singularities at the junction of
$\Gamma_0$ with $\Gamma_{in}$ and $\Gamma_{out}$. To this end we reflect the weak solution in a
way which preserves the boundary conditions and then apply the classical elliptic theory to the 
extended solution of the Lam\'e system. The details are given in Lemma \ref{lem_Lame}.
To show the regularity of the temperature we can apply Lemma \ref{lem_Neumann} since
on $\Gamma_{in} \cup \Gamma_{out}$ the boundary condition on the density reduces to the Neumann condition.

\section{Bounds on the approximating sequence}  \label{sec_conv}
In this section we will show the bounds on the sequence (\ref{system_seq}). 
Due to the term $u \cdot \nabla \sigma$ in the continuity equation we are not able
to show directly the convergence of this sequence in $W^2_p \times W^1_p \times W^2_p$ to the strong solution of (\ref{system}).
We can show however its boundedness in $W^2_p \times W^1_p \times W^2_p$. Next, using this
bound we derive the Cauchy condition in $W^1_2 \times L_{\infty}(L_2) \times W^1_2$, and thus show the convergence 
in this space to some $(u,\sigma,\eta)$. On the other hand, the boundedness implies weak
convergence in $W^2_p \times W^1_p \times W^2_p$, and the limit must be $(u,\sigma,\eta)$; hence the solution is strong.

The following lemma gives the boundedness of $(u^n,\sigma^n,\eta^n)$ in $W^2_p \times W^1_p \times W^2_p$. 
\begin{lem}  \label{lem_seq_bound}
Let $\{(u^n,\sigma^n,\eta^n)\}$ be a sequence of solutions to (\ref{system_seq}) starting from 
$(u^0,\sigma^0,\eta^0)=([0,0,0],0,0)$. Then
\begin{equation}
\|u^n\|_{W^2_p} + \|\sigma^n\|_{W^1_p} + \|\eta^n\|_{W^2_p} \leq M,  \label{est_seq_bound}
\end{equation}
where $M$ can be arbitrarily small provided that $D_0$ defined in \eqref{D0}, quantities $\|B_i\|_{W^{1-1/p}_p(\Gamma)}$, $i=1,2$, $\|\sigma_{in}-1\|_{W^1_p(\Gamma_{in})}$
and $\|U\|_{W^2_p}$ are small enough and $\alpha$, $L$, $\kappa$ fulfill the assumptions of Theorem \ref{main_thm}. 
\end{lem}
\noindent
\begin{proof} Estimate (\ref{est_lin}) for  system (\ref{system_seq}) reads
\begin{equation}  \label{unwn_1}
\begin{array}{c}
\|u^{n+1}\|_{W^2_p} + \|\sigma^{n+1}\|_{W^1_p} + \|\eta^{n+1}\|_{W^2_p} \leq \\
\leq C \, \big[ \|F(u^n,\sigma^n,\eta^n)\|_{L_p} + \|G(u^n,\sigma^n,\eta^n)\|_{W^1_p} + \|H(u^n,\sigma^n,\eta^n)\|_{L_p}\\
 + \sum_{i=1}^2\|B_i\|_{W^{1-1/p}_p(\Gamma)} + \|\sigma_{in}\|_{W^1_p(\Gamma_{in})} \big].
\end{array}
\end{equation}
Denoting $A_n = \|u^n\|_{W^2_p} + \|\sigma^n\|_{W^1_p} + \|\eta^n\|_{W^2_p}$, 
from (\ref{est_FGH}) and (\ref{unwn_1}) we get
\begin{equation} \label{an1}
A_{n+1} \leq C(A_n^2 + A_n^3 + D_0).
\end{equation}
We aim at showing that $A_n$
thus $A_n$ is bounded by a constant that can be arbitrarily small
provided that $A_0$ and $D_0$ are small enough.
Indeed, let us fix $0 < \delta < \frac{1}{8C}$. (Note that the constant $C$ can be without loss of generality taken larger than $1$.)  Assume that $CD_0<\delta$.
Then (\ref{an1}) entails an implication $A_n \leq 2\delta \Rightarrow A_{n+1} \leq 2\delta$
and we can conclude that
%
\begin{equation}
\|u^n\|_{W^2_p} + \|\sigma^n\|_{W^1_p} + \|\eta^n\|_{W^2_p} \leq 2 \delta \quad \forall \, n \in \mathbb{N}. 
\end{equation}
\end{proof}

The next lemma almost completes the proof of the Cauchy condition in $W^1_2 \times L_{\infty}(L_2) \times W^1_2$
for the iterating scheme.
\begin{lem} \label{lem_cauchy1}
Let the assumptions of Lemma \ref{lem_seq_bound} hold. Then we have
\begin{equation}  \label{est_cauchy1}
\begin{array}{c}
\|u^{n+1}-u^{m+1}\|_{W^1_2} + \|\sigma^{n+1}-\sigma^{m+1}\|_{L_{\infty}(L_2)} + \|\eta^{n+1}-\eta^{m+1}\|_{W^1_2} \\
\leq E(M) \, \big( \|u^{n}-u^{m}\|_{W^1_2} + \|\sigma^{n}-\sigma^{m}\|_{L_{\infty}(L_2)} + \|\eta^{n}-\eta^{m}\|_{W^1_2} \big),
\end{array}
\end{equation}
where $M$ is the constant from (\ref{est_seq_bound}) and $E(M)$ can be taken arbitrarily small.
\end{lem}

\noindent
\begin{proof}
Subtracting (\ref{system_seq})$_m$ from (\ref{system_seq})$_n$ we arrive at
\begin{equation} \label{dif_seq1}
\begin{array}{c}
\partial_{x_1} (u^{n+1} - u^{m+1}) - {\bf S}(\nabla (u^{n+1}-u^{m+1})) + p_1 \nabla (\sigma^{n+1} - \sigma^{m+1}) \\
 + p_2\nabla (\eta^{n+1}-\eta^{m+1}) = F(u^n,\sigma^n,\eta^n) - F(u^m,\sigma^m,\eta^m),
\end{array}
\end{equation}
\begin{equation} \label{dif_seq2}
\begin{array}{c}
{\rm div}\, (u^{n+1}-u^{m+1}) + \partial_{x_1} (\sigma^{n+1} - \sigma^{m+1}) + (u^n+u_0) \cdot \nabla (\sigma^{n+1}-\sigma^{m+1}) = \\
= G(u^n,\sigma^n) - G(u^m,\sigma^m) - (u^n-u^m) \cdot \nabla \sigma^{m+1},
\end{array}
\end{equation}
\begin{equation} \label{dif_seq3}
\begin{array}{c}
r_1 \partial_{x_1}(\eta^{n+1} - \eta^{m+1}) + r_2 {\rm div}\,(u^{n+1} - u^{m+1})  -\kappa \Delta (\eta^{n+1} - \eta^{m+1}) = \\
= H(u^n,\sigma^n,\eta^n)-H(u^m,\sigma^m,\eta^m) 
\end{array}
\end{equation} 
\begin{equation} \label{dif_seq4}
\begin{array}{c}
{\bf S}(\nabla(u^{n+1}-u^{m+1}))\cdot \tau_i +\alpha \, (u^{n+1}-u^{m+1}) \cdot \tau_i|_{\Gamma} = 0, \\
n\cdot  (u^{n+1}-u^{m+1})|_{\Gamma} = 0, \\
\sigma^{n+1}-\sigma^{m+1}|_{\Gamma_{in}}=0,\\
\kappa \frac{\partial (\eta^{n+1}-\eta^{m+1})}{\partial n} + L (\eta^{n+1}-\eta^{m+1}) =0.
\end{array}
\end{equation}
Estimate (\ref{ene}) applied to this system yields
\begin{displaymath}
\begin{array}{c}
\|u^{n+1}-u^{m+1}\|_{W^1_2} + \|\sigma^{n+1}-\sigma^{m+1}\|_{L_{\infty}(L^2)} + \|\eta^{n+1}-\eta^{m+1}\|_{W^1_2} \leq \nonumber\\
\|F(u^n,\sigma^n,\eta^n) - F(u^m,\sigma^m,\eta^m)\|_{V^*} + \|G(u^n,\sigma^n) - G(u^m,\sigma^m)\|_{L_2} \\ + \|(u^n - u^m) \cdot \nabla \sigma^{m+1}\|_{L_2}
+ \|H(u^n,\sigma^n,\eta^n) - H(u^m,\sigma^m,\eta^m)\|_{L_2}.
\end{array}
\end{displaymath}
In order to derive (\ref{est_cauchy1}) from the above inequality we have to examine the r.h.s.
The differences in $G$ and $H$ are bounded in a straightforward way: using the imbedding
$W^1_p \hookrightarrow L_{\infty}$ and H\"older's inequality we derive
\begin{equation}
\begin{array}{c}
\|G(u^n,\sigma^n) - G(u^m,\sigma^m)\|_{L_2} + \|H(u^n,\sigma^n,\eta^n) - H(u^m,\sigma^m,\eta^m)\|_{L_2}  \\ 
\leq E(M) \, \big( \|u^n-u^m\|_{W^1_2} + \|\sigma^n-\sigma^m\|_{L_{\infty}(L_2)} + \|\eta^n-\eta^m\|_{W^1_2} \big).\\
\end{array}
\end{equation}
\noindent
The difference in $F$ must be investigated more carefully.
A direct calculation yields
$$F(u^n,\sigma^n,\eta^n) - F(u^m,\sigma^m,\eta^m) = F^{n,m}_1 + F^{n,m}_2,$$ 
where in $F^{n,m}_1$ we include all the terms which do not contain $\nabla (\sigma^n-\sigma^m)$ as well as other terms containing the gradient of density or its difference.
Direct calculation using the imbedding $H^1 \hookrightarrow L_6$, $W^1_p \hookrightarrow L_{\infty}$ and H\"older's inequality
yields
\begin{equation}
\|F^{n,m}_1\|_{V^*} \leq E(M) \, \big( \|u^n-u^m\|_{W^1_2} + \|\sigma^n-\sigma^m\|_{L_{\infty}(L_2)} \big).
\end{equation}
Next,
\begin{equation} \label{fnm2}
\begin{array}{c}
F^{n,m}_2 = -(\partial_\rho \pi(\sigma^n+1,\overline{\theta}_0 + \overline{\theta}_1 + \eta^n)-p_1) \nabla (\sigma^n -  \sigma^m)  \\
- \big(\partial_\rho \pi(\sigma^n+1,\overline{\theta}_0 + \overline{\theta}_1+ \eta^n) - \partial_\rho \pi(\sigma^m+1,\overline{\theta}_0 + \overline{\theta}_1 + \eta^m)\big) \nabla \sigma^m \\
=: F^{n,m}_{2,1} + F^{n,m}_{2,2}.
\end{array}
\end{equation}
We have to compute $V^*$ norm of $F^{n,m}_2$, hence we multiply by $v \in V$ and integrate.
With the first term we have
$$
\begin{array}{c}
\big|\int_\Omega -(\partial_\rho \pi(\sigma^n+1,\overline{\theta}_0 + \overline{\theta}_1 + \eta^n)-p_1) \nabla (\sigma^n-\sigma^m) \cdot v \,dx\big| \\
\leq  \big|\int_\Omega (\partial_\rho \pi(\sigma^n+1,\overline{\theta}_0 + \overline{\theta}_1 + \eta^n)-p_1)  (\sigma^n-\sigma^m) {\rm div} \,v \,dx \big| \\
+ \big|\int_\Omega \nabla \partial_\rho \pi (\sigma^n+1,\overline{\theta}_0 + \overline{\theta}_1 + \eta^n) \cdot v (\sigma^n-\sigma^m) \, dx\big| \\
\leq C \big(\|\eta^n\|_{W^1_p} + \|\sigma^n\|_{W^1_p} + \|\overline{\theta_0}-T_0\|_{W^1_p} + \|\overline{\theta_1}\|_{W^1_p}\big) \|(\sigma^n-\sigma^m)\|_{L_2} \|v\|_{W^1_2},
\end{array}
$$
hence
\begin{equation}
\|F^{n,m}_{2,1}\|_{V^*} \leq E(M) \|(\sigma^n-\sigma^m)\|_{L_{\infty}(L_2)}.
\end{equation}
We have used the fact that $\frac{1}{2}+\frac{1}{6}+\frac{1}{p}<1$ and the imbedding $W^1_2 \hookrightarrow L^6$.
%
%
%
%
%
%
The other term $F^{n,m}_{2,2}$ can be estimated similarly, only without integration by parts.
Combining the estimates on $F^{n,m}_1$,$F^{n,m}_{2,1}$ and $F^{n,m}_{2,2}$ we conclude
\begin{equation} \label{est_fnm}
\|F(u^n,\sigma^n,\eta^n) - F(u^m,\sigma^m,\eta^m)\|_{V^*} 
\leq E(M) \big( \|u^n-u^m\|_{W^1_2} + \|\sigma^n-\sigma^m\|_{L_{\infty}(L_2)} + \|\eta^n-\eta^m\|_{W^1_2} \big).
\end{equation}
The only term that remains to estimate is $(u^n-u^m) \cdot \nabla \sigma^m$.
We emphasize that this is the term which makes it impossible to show the convergence
in $W^2_p \times W^1_p$ directly. Namely, if we would like to apply the estimate (\ref{est_lin})
to the system for the difference then we would have to estimate $\|(u^n-u^m) \cdot \nabla \sigma^m\|_{W^1_p}$
what can not be done as we do not have any knowledge about $\|\sigma^m\|_{W^2_p}$.
Fortunately we only need the $L_2$-norm of this awkward term, which can be bounded in a direct way as
\begin{equation}
\|(u^n-u^m) \cdot \nabla \sigma^m\|_{L_2} \leq \|u^n-u^m\|_{L_q} \, \|\nabla \sigma^m\|_{L_p}
\leq C \, \|\sigma^m\|_{W^1_p} \, \|u^n-u^m\|_{W^1_2},
\end{equation}
since $q = \frac{2p}{p-2} <6$ for $p<3$. We have thus completed the proof of (\ref{est_cauchy1}). \end{proof}

Now, Lemma \ref{lem_seq_bound} implies that the constant $E(M) < 1$ provided that
the data is small enough and the starting point $(u^0,\sigma^0,\eta^0) = ([0,0,0],0,0)$.
It completes the proof of the Cauchy condition in $H^1 \times L_{\infty}(L_2) \times H^1$ 
for the sequence $(u^n,\sigma^n,\eta^n)$.

{\bf Remark.} Lemmas \ref{lem_seq_bound} and \ref{lem_cauchy1} hold for any starting point $(u^0,\sigma^0,\eta^0)$
small enough in $W^2_p \times W^1_p \times W^2_p$, but we can start the iteration from $([0,0,0],0,0)$
without loss of generality.

\section{Proof of Theorem \ref{main_thm}} \label{sec_proof}
In this section we prove our main result, Theorem \ref{main_thm}. First we show existence of the solution
passing to the limit with the sequence $(u^n,\sigma^n,\eta^n)$ and next we show that this solution is unique
in the class of solutions satisfying (\ref{est_main}).

{\bf Existence of the solution}.
Since we have the Cauchy condition on the sequence $(u^n,\sigma^n,\eta^n)$ only in the space
$W^1_2 \times L_{\infty}(L_2) \times W^1_2$, first we have to show the convergence in the weak formulation
of the problem (\ref{system}). The sequence $(u^n,\sigma^n,\eta^n)$ satisfies in particular the
following weak formulation of (\ref{system_seq}):

\begin{equation} \label{weak1_seq}
\begin{array}{c}
\int_{\Omega} \{ v \cdot \partial_{x_1} u^{n+1} + {\bf S}(\nabla u^{n+1}): \nabla v
- p_1 \sigma^{n+1} \, {\rm div}\,v - p_2 \eta^{n+1} {\rm div}\,v \} \,dx\\
+ \int_{\Gamma} \alpha (u^{n+1} \cdot \tau_i) \, (v \cdot \tau_i) \,dS = 
 \int_{\Omega} F(u^n,\sigma^n,\eta^n) \cdot v \,dx + \sum_{i=1}^2\int_{\Gamma} B_i (v \cdot \tau_i) \,dS \quad \forall \; v \in V
\end{array}
\end{equation}

\begin{equation} \label{weak2_seq}
-\int_{\Omega} \sigma^{n+1} \tilde u^n \cdot \nabla \phi \,dx - \int_{\Omega} \sigma^{n+1} \phi \, {\rm div}\, \tilde u^{n+1} \,dx =
\int_{\Omega} \phi (G(u^n,\sigma^n) - {\rm div}\, u^{n+1}) \,dx + \int_{\Gamma_{in}} \sigma_{in} \phi \, dS
\end{equation}
$\forall \; \phi \in  C^{\infty}(\overline{\Omega}): \phi|_{\Gamma_{out}}=0$, where
$\tilde u^n = [1+(u^n+u_0)^{1},(u^n+u_0)^{2},(u^n+u_0)^{3}]$, and

\begin{equation} \label{weak3_seq}
\begin{array}{c}
(1+ e_2)\int_{\Omega} \partial_{x_1}\eta^{n+1} w + T_0 p_2\int_{\Omega} w {\rm div}\,u^{n+1} + \kappa \int_{\Omega} \nabla w \cdot \nabla \eta^{n+1} 
\\
+ \int_{\Gamma} Lw\eta^{n+1} \,dS = \int_{\Omega} H(u^n,\sigma^n,\eta^n)w\,dx
\end{array}
\end{equation}
$\forall \; w \in H^1$.

Now using the convergence in $W^1_2 \times L_{\infty}(L_2) \times W^1_2$ combined with the bound (\ref{est_seq_bound})
in $W^2_p \times W^1_p \times W^2_p$ we can pass to the limit in (\ref{weak1_seq})--(\ref{weak3_seq}).
The convergence of the l.h.s. of (\ref{weak1_seq})--(\ref{weak3_seq}) is obvious.
Recalling the definition (\ref{FGH}) of $F(\cdot),G(\cdot)$ and $H(\cdot)$ we verify easily that 
\begin{equation}  \label{conv_G} 
\int_{\Omega} G(u^n,\sigma^n) \cdot v \,dx \rightarrow \int_{\Omega} G(u,\sigma) \cdot v \,dx
\end{equation}
and
\begin{equation} \label{conv_H} 
\int_{\Omega} H(u^n,\sigma^n,\eta^n) \cdot v \,dx \rightarrow \int_{\Omega} H(u,\sigma,\eta) \cdot v \,dx,
\end{equation}
and the only step in showing the convergence of $F$ which requires more attention is to show that
\begin{equation} \label{6.5a}
\int_{\Omega} (\partial_\rho \pi(\sigma^n+1,\overline{\theta}_0 + \overline{\theta}_1 + \eta^n)-p_1) \nabla \sigma^n \cdot v \,dx \to  \int_{\Omega} (\partial_\rho \pi(\sigma+1,\overline{\theta}_0 + \overline{\theta}_1 + \eta)-p_1) \nabla \sigma \cdot v \,dx.
\end{equation}
However, since $\sigma^n \to \sigma$ and $\eta^n \to \eta$ in $C(\overline{\Omega})$ as well as $\nabla \sigma ^n \rightharpoonup \nabla \sigma$ in $L_p(\Omega)$, we easily verify that \eqref{6.5a} holds true.
%
We conclude that $(u,\sigma,\eta)$ satisfies (\ref{weak1_seq})--(\ref{weak3_seq}). Now we need to show 
that this implies the strong formulation (\ref{system}), what can be done in a standard way, just integrating by parts in the weak formulation.

Now we set 
$$
v = \bar v + u + u_0, \quad \rho=\sigma+1, \quad \theta = \eta + \overline{\theta}_0 + \overline{\theta}_1
$$  
and we see that $(v,\rho,\theta)$ solves (\ref{main_system}). 
We clearly have $E(M) = E(D_0)$ where $D_0$ is defined in (\ref{D0}),
hence estimate (\ref{est_main}) holds.  

{\bf Uniqueness.} The uniqueness in the class of solutions satisfying (\ref{est_main})
actually results directly from the method of the proof, more precisely from the proof
of (\ref{est_cauchy1}). Namely, we can show uniqueness on the level of perturbations,
i.e. uniqueness for (\ref{system}). Consider two solutions with the same data, denote it by
$(u_1,\sigma_1,\eta_1)$ and $(u_2,\sigma_2,\eta_2)$. Their difference then satisfy
\begin{displaymath}
\begin{array}{c}
\partial_{x_1} (u_1 - u_2) - {\bf S}(\nabla u_1 - \nabla u_2) 
+ p_1 \nabla (\sigma_1 - \sigma_2) \\
+ p_2\nabla (\eta_1-\eta_2) = F(u_1,\sigma_1,\eta_1) - F(u_2,\sigma_2,\eta_2),
\end{array}
\end{displaymath}
\begin{displaymath}
\begin{array}{c}
{\rm div}\, (u_1-u_2) + \partial_{x_1} (\sigma_1 - \sigma_2) + (u_1+u_0) \cdot \nabla (\sigma_1-\sigma_2)  \\
= G(u_1,\sigma_1) - G(u_2,\sigma_2) - (u_1-u_2) \cdot \nabla \sigma_2,
\end{array}
\end{displaymath}
\begin{displaymath}
\begin{array}{c}
(1+ e_2)\partial_{x_1}(\eta_1 - \eta_2) + T_0 p_2 {\rm div}\,(u_1 - u_2)  -\kappa \Delta (\eta_1 - \eta_2) \\
= H(u_1,\sigma_1,\eta_1)-H(u_2,\sigma_2,\eta_2) 
\end{array}
\end{displaymath}
\begin{displaymath}
\begin{array}{c}
{\bf S}(\nabla(u_1-u_2)n)\cdot \tau_i +\alpha \, (u_1-u_2) \cdot \tau_i|_{\Gamma} = 0, \\
n\cdot  (u_1-u_2)|_{\Gamma} = 0, \\
\sigma_1-\sigma_2|_{\Gamma_{in}}=0,\\
\kappa \frac{\partial (\eta_1-\eta_2)}{\partial n} + L (\eta_1-\eta_2) =0.
\end{array}
\end{displaymath}  

\noindent
Now recall that to show (\ref{est_cauchy1}) we applied only (\ref{dif_seq1})--(\ref{dif_seq4})
and (\ref{est_seq_bound}), hence from the above equations
we conclude
\begin{equation}
\begin{array}{c}
\|u_1-u_2\|_{W^1_2} + \|\sigma_1-\sigma_2\|_{L_{\infty}(L_2)} + \|\eta_1-\eta_2\|_{W^1_2} \\
\leq E(M) \, \big( \|u_1-u_2\|_{W^1_2} + \|\sigma_1-\sigma_2\|_{L_{\infty}(L_2)} + \|\eta_1-\eta_2\|_{W^1_2} \big).
\end{array}
\end{equation}
Provided the data are small enough we have $E(M)<1$ and so
$$
\|u_1-u_2\|_{W^1_2} + \|\sigma_1-\sigma_2\|_{L_{\infty}(L_2)} + \|\eta_1-\eta_2\|_{W^1_2} = 0
$$
which completes the proof of uniqueness and hence of Theorem \ref{main_thm}. 

\medskip

\noindent \emph{Acknowledgment:}
 The work of  M.P. was supported by the grant of
the Czech Science Foundation No. 201/09/0917.

\end{document}